\documentclass{amsart}
\usepackage{amssymb}
\usepackage{amsmath}
\usepackage{mathrsfs}   
\usepackage[all]{xy}
\usepackage{stmaryrd}  
\usepackage{multirow}  
\usepackage{bigdelim}  

\renewcommand{\subsection}{%
  \removelastskip%
  \vskip.5\baselineskip%
  \refstepcounter{subsection}%
  \noindent%
  {\bf (\thesubsection)} }

\newtheorem{theorem}{Theorem}[subsection]
\newtheorem{lemma}[theorem]{Lemma}
\newtheorem{corollary}[theorem]{Corollary}
\newtheorem{proposition}[theorem]{Proposition}

\newtheorem{conjecture}[theorem]{Conjecture}

\theoremstyle{definition}

\theoremstyle{remark}
\newtheorem{remark}[theorem]{\it Remark}

\def\Z{\mathbb{Z}}
\def\C{\mathbb{C}}
\def\Q{\mathbb{Q}}
\def\F{\mathbb{F}}
\def\A{\mathbb{A}}
\def\T{\mathbb{T}}

\def\ie{{\it i.e.}}
\def\eg{{\it e.g.}}

\def\cf{{\it c.f.}}

\let\ms\mathscr
\let\mf\mathfrak
\let\mc\mathcal

\let\wt\widetilde
\let\ol\overline

\let\bs\backslash
\let\lbb\llbracket
\let\rbb\rrbracket

\DeclareMathOperator{\tr}{tr}
\DeclareMathOperator{\Res}{Res}
\DeclareMathOperator{\ad}{ad}
\DeclareMathOperator{\bN}{\mathbf{N}}
\DeclareMathOperator{\End}{End}
\DeclareMathOperator{\Hom}{Hom}

\DeclareMathOperator{\Spec}{Spec}
\DeclareMathOperator{\re}{Re}
\DeclareMathOperator{\im}{im}
\DeclareMathOperator{\Gal}{Gal}
\DeclareMathOperator{\Ind}{Ind}

\DeclareMathOperator{\Br}{Br}

\newcommand{\GL}{\mathrm{GL}}
\newcommand{\PGL}{\mathrm{PGL}}

\newcommand{\sfrac}[2]{{\textstyle \frac{#1}{#2}}}

\newcommand{\aux}{\mathrm{aux}}
\newcommand{\TW}{\mathrm{TW}}
\newcommand{\Frob}{\mathrm{Frob}}
\newcommand{\loc}{\mathrm{loc}}
\newcommand{\ram}{\mathrm{ram}}

\newcommand{\dR}{\mathrm{dR}}

\def\mat#1#2#3#4{\left( \begin{array}{cc} #1 & #2 \\ #3 & #4 \end{array} \right)}

\title[On two dimensional weight two odd representations of totally real fields]
{On two dimensional weight two odd\\representations of totally real fields}
\author{Andrew Snowden}
\date{May 26, 2009.}

\begin{document}

\begin{abstract}
We say that a two dimensional $p$-adic Galois representation $G_F \to
\GL_2(\ol{\Q}_p)$ of a number field $F$ is \emph{weight two} if it is de Rham
with Hodge-Tate weights 0 and $-1$ equally distributed at each place above
$p$; for example, the Tate module of an elliptic curve has this property.  The
purpose of this paper is to establish a variety of results concerning odd
weight two representations of totally real fields in as great a generality as
we are able.  Most of these results are improvements upon existing results.
Three of our main results are as
follows.  (1) We prove a modularity lifting theorem for odd weight two
representations, extending a theorem of Kisin to include representations
which are not potentially crystalline.  (2) We show
that essentially any odd weight two representation is potentially modular,
following the ideas of Taylor.
(3) We show that one can lift essentially any odd residual representation to a
minimally ramified weight two $p$-adic representation, using some ideas of
Khare-Wintenberger.  As an application of these results we show that if $\rho$
is a sufficiently irreducible odd weight two $p$-adic representation of a
totally real field $F$ and
either $F$ has odd degree or $\rho$ is indecomposable at some finite
place $v \nmid p$ then $\rho$ occurs as the Tate module of a $\GL_2$-type
abelian variety.  This establishes some new cases of the Fontaine-Mazur
conjecture.
\end{abstract}

\maketitle
\tableofcontents

\section{Introduction}

\subsection
Fix a totally real field $F$ and a prime $p$.  Denote by
$G_F$ the absolute Galois group of $F$.  We say that a representation
$G_F \to \GL_2(\ol{\Q}_p)$ is \emph{weight two} if it is de Rham with
non-positive Hodge-Tate weights at each place above $p$ and has determinant
equal to a finite order character times the cyclotomic character;
see \S \ref{ss:wt2} for more on this definition.  The purpose of this paper is
to establish a variety of results concerning odd weight two representations
of totally real fields in as great as
generality as we are able.  These results are, for the most part, extensions
of the work of Khare, Kisin, Taylor and Wintenberger.

For the purposes of the introduction, assume $p \ne 2, 5$.  We prove many
of the results outlined below when $p=5$ with an additional hypothesis; we
do not consider $p=2$ in this paper.  One of our main results is the following
theorem:

\begin{theorem}
\label{mainthm}
Let $\rho:G_F \to \GL_2(\ol{\Q}_p)$ be a finitely ramified, odd, weight two
representation such that $\ol{\rho} \vert_{G_{F(\zeta_p)}}$ is
irreducible.  Then $\rho$ is potentially modular.
\end{theorem}

Here ``finitely ramified'' means $\rho$ ramifies at only finitely many
places; ``odd'' means $\det{\rho(c)}=-1$ for any complex conjugation $c$;
$\ol{\rho}$ denotes the mod $p$ reduction of $\rho$; and
``potentially modular'' means that there is a finite, totally real extension
$F'/F$ such that $\rho \vert_{G_{F'}}$ is associated to a \emph{cuspidal}
Hilbert eigenform.  (All modular forms we use will be cuspidal.)
We actually prove a more precise version of Theorem~\ref{mainthm} ---  see
Theorem~\ref{thm:pmod2}.  Using known consequences of potential modularity
we obtain the following:

\begin{corollary}
\label{maincor}
Let $\rho$ be as in the theorem.
\begin{enumerate}
\item If $\rho$ is unramified at $v$ then the eigenvalues of $\rho(\Frob_v)$
belong to $\ol{\Q} \subset \ol{\Q}_p$ and under any embedding into $\C$ have
modulus $(\bN{v})^{1/2}$.
\item The representation $\rho$ fits into a strongly compatible system.
\item For any isomorphism $i:\ol{\Q}_p \to \C$ the $L$-function $L(s, \rho, i)$
converges to a non-zero holomorphic function for $\re{s}>\sfrac{3}{2}$, has
a meromorphic continuation to the entire complex plane and
satisfies the expected functional equation.
\item Assume that either $F$ has odd degree or for some place $v \nmid p$ the
representation $\rho \vert_{G_{F_v}}$ is indecomposable.  Then $\rho$ occurs
as the Tate module of a $\GL_2$-type abelian variety.  In particular, the
Fontaine-Mazur conjecture holds for $\rho$.
\end{enumerate}
\end{corollary}

This corollary is proved in \S \ref{s:conseq}.
The additional hypothesis in the final statement stems from the fact that the
Galois representation attached to a parallel weight two Hilbert eigenform over
a totally real field is known to come from an abelian variety only if the the
field has odd degree or the form is square integrable at some finite place.

Another of our main results is the following:

\begin{theorem}
\label{mainthm2}
Let $\ol{\rho}:G_F \to \GL_2(\ol{\F}_p)$ be an irreducible odd representation.
\begin{enumerate}
\item Assume $\ol{\rho} \vert_{G_{F(\zeta_p)}}$ is irreducible.  Then there
exists a minimally ramified weight two lift of $\ol{\rho}$ to $\ol{\Q}_p$.
\item The representation $\ol{\rho}$ occurs as the $\mf{p}$-torsion of a
$\GL_2$-type abelian variety.
\end{enumerate}
\end{theorem}

The first statement is restated and proved as Proposition~\ref{prop:minlift};
the second is restated and proved as Proposition~\ref{prop:astors}.
In \S \ref{s:problems} we will prove a result considerably stronger
than the first statement, showing that one has
great control over the local behavior of lifts:  one can construct lifts
with a given action of inertia at finitely many places, so long as there
are no obvious obstructions.  In Remark~\ref{rem:noA1} we explain how one can
often remove the additional hypothesis in the first statement.

We mention one more of our results here, a modularity lifting theorem
(restated and proved as Theorem~\ref{thm:mlt}):

\begin{theorem}
\label{mainthm3}
Let $\rho:G_F \to \GL_2(\ol{\Q}_p)$ be a finitely ramified, odd, weight two
representation such that $\ol{\rho} \vert_{G_{F(\zeta_p)}}$ is
irreducible.  Assume that
there exists a parallel weight two Hilbert eigenform $f$ such that $\ol{\rho}=
\ol{\rho}_f$ and that $\rho$ and $\rho_f$ have the same ``type'' at each place
above $p$, where ``type'' is one of:  potentially crystalline ordinary,
potentially crystalline non-ordinary or not potentially crystalline.  Then
there is a Hilbert eigenform $g$ such that $\rho=\rho_g$.
\end{theorem}

\subsection\label{ss:wt2}
We now give a precise definition of the term ``weight two.''
Let $F_v/\Q_p$ be a finite extension and let $E/\Q_p$ be any extension.
We say that a representation $\rho_v:G_{F_v} \to \GL_2(E)$ is \emph{weight two}
if the $(E \otimes_{\Q_p} F_v)$-module $(\rho \otimes_{\Q_p} B_{\dR})^{
G_{F_v}}$ is free of rank two and its associated graded is free of rank one in
degrees 0 and $-1$.  Equivalently, $\rho_v$ is weight two if it
is de Rham with non-positive Hodge-Tate weights and $\det{\rho_v}$ restricted
to inertia is a finite order character times the cyclotomic character.  We
use the convention that the cyclotomic character $\chi_p$ has Hodge-Tate
weight $-1$.

For a global field $F/\Q$, we say that a representation $\rho:G_F \to
\GL_2(E)$ is \emph{weight two} if it is at every place above $p$.
Equivalently, $\rho$ is weight two if $\rho \vert_{G_{F_v}}$ is de Rham with
non-positive Hodge-Tate weights for each $v \mid p$ and $\det{\rho}$ is equal
to a finite order character times the cyclotomic character.
We use this terminology because the
Galois representations coming from parallel weight two Hilbert eigenforms have
this property.

\subsection
We now give an overview of the structure of the paper.
\begin{itemize}
\item In \S \ref{s:local} we give some results on deformation rings of
representations of local Galois groups.  These results are mostly due to Kisin
\cite{Kisin}, though we will also need some results of Gee \cite{Gee} and
some results that do not seem to be covered by either author (for which we
provide proofs).
\item In \S \ref{s:reqt} we prove an $R=\T$ theorem using the method of
Taylor-Wiles, as improved by several authors.  We follow \cite{Kisin} very
closely in our treatment.  This theorem relies on the local results of
\S \ref{s:local}.
\item In \S \ref{s:modlift} we establish
Theorem~\ref{mainthm3}.  We deduce this from the $R=\T$ theorem of
\S \ref{s:reqt} using the base change tricks of \cite{Kisin}.  At the end
of \S \ref{s:modlift} we appeal to the results of \S \ref{s:problems} and the
modularity lifting theorem of Skinner-Wiles to produce some mod $p$
congruences between Hilbert modular forms which are special at $p$ and those
which are not.
\item In \S \ref{s:pmod} we prove our potential modularity results
(these are somewhat stronger than the statement in Theorem~\ref{mainthm}).
Such results were first established by Taylor \cite{Taylor3}, and our
proofs do not differ much from his.  In fact, our proofs are less difficult
because we have stronger modularity lifting theorems available.  The main
tools used in the proofs are Theorem~\ref{mainthm3} and a theorem of
Moret-Bailly.
\item In \S \ref{s:defring} we prove that certain global deformation
rings are finite over $\Z_p$ of non-zero rank, and therefore have
characteristic zero points.  Such results were first established by
Khare and Wintenberger \cite{KhareWintenberger}, and our proof does not
differ much from theirs.  The proof has two parts.  First one uses purely
Galois theoretic results to show that the deformation rings have non-zero
Krull dimension.  Then one uses the potential modularity results of
\S \ref{s:pmod} together with the finiteness of deformation rings
of modular representations (proved in \S \ref{s:reqt}) to
conclude that the deformation rings are finite.
\item In \S \ref{s:problems} we use the results of \S \ref{s:defring} to
produce lifts of residual representations with prescribed behavior on
inertia at finitely many places.  We learned the method of proof from a
paper of Gee \cite{Gee2}.  Our results improve on the ones
there in two ways:  first, we make no modularity
hypothesis; and second, we can control the monodromy operator of lifts.
To apply the results of \S \ref{s:defring} for these purposes, we need
some more results on local deformation rings.  These are due to Kisin
\cite{Kisin4} and Gee \cite{Gee2}, though as before we need to establish
some new results in certain easy cases.  We also prove some local
lifting results which do not appear to be in the literature.
\item In \S \ref{s:additional} we prove two miscellaneous results which
rely on all the preceding work.  The first is a descent result for mod
$p$ Hilbert eigenforms.  A similar result, though less general, was
obtained previously by Khare \cite{Khare}.  The second result is an
improvement on
our potential modularity result.  It states that one can achieve modularity
by passing to an extension which splits at a finite prescribed set of
primes.  We establish this by using a simple modification of the theorem
of Moret-Bailly and our descent theorem for mod $p$ eigenforms.
\item In \S \ref{s:conseq} we prove Corollary~\ref{maincor}.  The proofs of
the statements in this corollary are, for the most part, well-known or
appear already in the literature.  It seems worthwhile, though, to have them
consolidated in one place under one notational
scheme.  Statement (1) of the corollary follows easily from the corresponding
result for modular representations, which is due to Blasius \cite{Blasius}.
The proof of statement (2) is due to Dieulefait \cite{Dieulefait}.  The
proofs of statements (3) and (4) are due to Taylor \cite{Taylor},
\cite{Taylor3}.  We should point out, though, that to prove (4) in the stated
generality requires some new arguments, in particular the strengthened form
of potential modularity given in \S \ref{s:additional}.
\end{itemize}

\subsection
We now list some notation that we will use throughout the paper.
\begin{itemize}
\item The letter $p$ always denotes an odd prime.  We deal with $p$-adic or
mod $p$ representations.
\item The symbol $F$ denotes the field over which we work.  It is almost always
totally real and there is usually a representation of its absolute Galois
group $G_F$ that we are studying.
\item Finite extensions of $\Q_p$ or $\Q_{\ell}$ whose Galois group we study
 will be denoted by something
like $F_v$.  Often $F_v$ will be the completion of a number field $F$ at a
place $v$; when this is not the case, $v$ is just a decoration to remind the
reader that $F_v$ is a local field.  We denote by $G_{F_v}$ the absolute
Galois group of $F_v$ and by $I_{F_v}$ the inertia subgroup.
\item For a prime number $p$ we denote by $\Sigma_{F, p}$ or $\Sigma_p$ the
set of places of $F$ above $p$.  We also write $\Sigma_{F, \infty}$ or
$\Sigma_{\infty}$ for the set of infinite places of $F$.
\item We let $E/\Q_p$ be an extension, typically finite, with ring of integers
$\ms{O}$ and residue field $k$.  All our deformation theory takes place on the
category of local $\ms{O}$-algebras.
\item We use $\rho$ (resp.\ $\ol{\rho}$) to denote $p$-adic (resp.\ mod $p$)
representations of $G_F$.  We typically denote representations of $G_{F_v}$
by $\rho_v$ or $\ol{\rho}_v$.
\item We denote by $\chi_p$ the $p$-adic cyclotomic character and by
$\ol{\chi}_p$ its reduction modulo $p$.
\item The symbol $\Frob$ will always denote an arithmetic Frobenius.
\end{itemize}

\subsection
I would like to thank several people for the help they gave me:
Bhargav Bhatt, Brian Conrad, Mark Kisin, Stefan Patrikis, Chris Skinner,
Jacob Tsimerman and Andrew Wiles.
Wiles suggested a problem to me that eventually lead to the writing of this
paper.  Patrikis helped me immensely, answering several questions I had about
Kisin's paper \cite{Kisin} and carefully reading and commenting on an
earlier version of this paper.

\section{Local deformation rings}
\label{s:local}

\subsection
Let $F_v/\Q_p$ be a finite extension and let $\rho_v:G_{F_v} \to \GL_2(E)$
be a weight two representation, where $E$ is any extension of $\Q_p$.  Let $V$
be the representation space of $\rho$.
\begin{itemize}
\item We say that $\rho_v$ is \emph{admissible of type $A$} if it is
crystalline, ordinary and has cyclotomic determinant.  This means that $\rho_v$
is crystalline and there is an exact sequence
\begin{displaymath}
0 \to E(\chi_p \psi) \to V \to E(\psi^{-1}) \to 0
\end{displaymath}
where $\psi$ is an unramified character.
\item We say that $\rho_v$ is \emph{admissible of type $B$} if it is
crystalline, non-ordinary and has cyclotomic determinant.
\item We say that $\rho_v$ is \emph{admissible of type $C$} if there is an
exact sequence
\begin{displaymath}
0 \to E(\chi_p) \to V \to E \to 0.
\end{displaymath}
Of course, this forces $\rho_v$ to have cyclotomic determinant.
\end{itemize}
We say that $\rho_v$ has \emph{type $\ast$} if there is some finite extension
$F'_v/F_v$ such that $\rho_v \vert_{G_{F_v'}}$ is admissible of type $\ast$.
As long as $\det{\rho_v}$ is a finite order character times the cyclotomic
character $\rho_v$ will have some type.  It is possible
for $\rho_v$ to have both type $A$ and $C$; we then say that $\rho$ has
type $A/C$.

\begin{remark}
Let $f$ be a parallel weight two Hilbert eigenform over a totally real number
field $F$, $\pi$ the corresponding automorphic representation and $\rho$ the
corresponding $p$-adic Galois representation.  Let $v$ be a place of $F$ over
$p$.  Then $\rho$ has type $A$ or $B$ (resp.\ $C$) at $v$ if and
only if $\pi_v$ is principal series or cuspidal (resp.\ special).  The
distinction between $A$ and $B$ amounts to a certain congruence for the Hecke
eigenvalue.  Note that $\rho \vert_{G_{F_v}}$ never has type $A/C$; in fact
this is true for any global representation, \cf\ Proposition~\ref{prop:type}.
\end{remark}

\subsection
We now assume that $F_v$ contains a non-trivial $p$th root of unity, so that
the $p$-adic cyclotomic character $\chi_p$ of $G_{F_v}$ reduces to the trivial
character.

\begin{proposition}
\label{prop:pring}
Let $\ol{\rho}_v:G_{F_v} \to \GL_2(k)$ be the trivial representation and let
$R^{\Box}$ denotes its universal framed deformation ring.  Let $\ast$ be $A$,
$B$ or $C$.  There is then a unique quotient $R^{\Box, \ast}$ of $R^{\Box}$
such that:
\begin{enumerate}
\item $R^{\Box, \ast} \otimes_{\ms{O}} \ms{O}_{E'}$ is a domain, for any
finite extension $E'/E$.
\item $R^{\Box, \ast}$ is flat over $\ms{O}$ of relative
dimension $[F_v:\Q_p]+3$.
\item $R^{\Box, \ast}[1/p]$ is formally smooth over $E$.
\item Let $E'/E$ be a finite extension.  A map $R^{\Box} \to E'$ factors
through $R^{\Box, \ast}$ if and only if the corresponding representation is
admissible of type $\ast$.
\end{enumerate}
\end{proposition}

\begin{proof}
If $\ast$ is $A$ or $B$ then this is just \cite[Corollary~2.5.16]{Kisin} (see
also the proof of \cite[Theorem~3.4.11]{Kisin}), with the improvement of
\cite{Gee} in the $\ast=B$ case.  We give a proof in the case $\ast=C$ below.
\end{proof}

\subsection
We now prove Proposition~\ref{prop:pring} for $\ast=C$.  We closely follow
\cite[\S 2.4]{Kisin2}.  Let $V_k$ be the trivial two dimensional
representation of $G_{F_v}$.  We denote by $\mf{Aug}_{\ms{O}}$ the category
of pairs $(A, I)$ where $A$ is an $\ms{O}$-algebra and $I \subset A$ is a
nilpotent ideal with $\mf{m}_{\ms{O}} A \subset I$.  Let $D^{\chi}$ be the
groupoid valued functor on $\mf{Aug}_{\ms{O}}$ which assigns to $(A, I)$ the
category of pairs $(V_A, \iota_A)$ where:
\begin{itemize}
\item $V_A$ is a free rank two $A$-module with a discrete action of $G_{F_v}$
with determinant $\chi_p$; and
\item $\iota_A:V_A \otimes_A A/I \to V_k \otimes_k A/I$ is an isomorphism of
$(A/I)[G_{F_v}]$-modules.
\end{itemize}
We define $D^C$ to be the groupoid valued functor which assigns to $(A, I)$
the category of triples $(V_A, L_A, \iota_A)$ where:
\begin{itemize}
\item $V_A$ is a free rank two $A$-module with a discrete action of $G_{F_v}$
with determinant $\chi_p$;
\item $L_A \subset V_A$ is an $A$-line (that is, $L_A$ is a rank one
projective $A$-submodule for which $V_A/L_A$ is projective) on which $G_{F_v}$
acts by $\chi_p$; and
\item $\iota_A:V_A \otimes_A A/I \to V_k \otimes_k A/I$ is an isomorphism of
$(A/I)[G_{F_v}]$-modules.
\end{itemize}
Note that $D^C$ is analogous to the functor $D^{\mathrm{ord}, \chi}_{V_{\F}}$
in \cite[\S 2.4]{Kisin2}.  We have the following analogue of
\cite[Proposition~2.4.4]{Kisin2}.

\begin{proposition}
We have the following:
\begin{enumerate}
\item There is a morphism of functors $D^C \to D^{\chi}$ taking $(V_A, L_A,
\iota_A)$ to $(V_A, \iota_A)$.  This morphism is relatively representable and
projective.
\item If $R$ is a complete local ring with residue field $k$, $\xi$ a map
$\Spec(R) \to D^{\chi}$ and $\ms{L}^C_{\xi}$ the projective $R$-scheme
representing $\Spec(R) \times_{D^{\chi}} D^C$ then the morphism
$\ms{L}^C_{\xi} \to \Spec(R)$ becomes a closed immersion after inverting $p$.
\item If $\xi:\Spec(R) \to D^{\chi}$ is formally smooth then $\ms{L}^C_{\xi}$
is formally smooth over $\ms{O}$.
\end{enumerate}
\end{proposition}

\begin{proof}
The proof is the same as that given in \cite[Proposition~2.4.4]{Kisin2},
except for one point in (3).  There, instead of
using \cite[Lemma~2.4.2]{Kisin2}, one uses the fact that if $M \to M'$ is a
surjection of $\Z_p$-modules on which $p$ is nilpotent then $H^1(G_{F_v},
M(\chi_p)) \to H^1(G_{F_v}, M'(\chi_p))$ is surjective (this follows easily
from Kummer theory).
\end{proof}

Let $D^{\Box, \chi}$ and $D^{\Box, C}$ be the framed versions of
$D^{\chi}$ and $D^C$.  Then $D^{\Box, \chi}$ is pro-representable and
$D^{\Box, \chi} \to D^{\chi}$ is formally smooth.  Let $R^{\Box, \chi}$ denote
the complete local $\ms{O}$-algebra representing $D^{\Box, \chi}$.  The above
proposition tells us that
$\ms{L}^{\Box, C}=R^{\Box, \chi} \times_{D^{\chi}} D^C$ is a projective
scheme over $\Spec(R^{\Box, \chi})$, is formally smooth over $\ms{O}$ and
the map $\ms{L}^{\Box, C} \to \Spec(R^{\Box, \chi})$ becomes a closed
embedding after inverting $p$.  We define $\Spec(R^{\Box, C})$ to be
the scheme theoretic image of $\ms{L}^{\Box, C} \to \Spec(R^{\Box, \chi})$.

We now verify that $R^{\Box, C}$ satisfies (1)--(4) of
Proposition~\ref{prop:pring}.  The fiber of
$\ms{L}^{\Box, C}$ over the closed point of $R^{\Box, C}$ is the projective
line and so the proof of \cite[Corollary~2.4.10]{Kisin} implies that
$\ms{L}^{\Box, C}$
is connected.  As the image of a connected smooth scheme is a connected
integral scheme, the ring $R^{\Box, C}$ is a domain.  Since the same argument
works when considering deformations to $\ms{O}_{E'}$ algebras and that functor
is represented by $R^{\Box, C} \otimes_{\ms{O}} \ms{O}_{E'}$, we conclude
that (1) holds.  For (3) note that
$\mc{L}^{\Box, C} \to \Spec(R^{\Box, C})$ is an isomorphism after inverting
$p$ and the first space is smooth after inverting $p$.  The proof of (4)
goes just like the proof in \cite[Corollary~2.4.5]{Kisin2}.

We now verify (2).  First note that there exists a ring homomorphism
$R^{\Box, C} \to \ms{O}$, as there
exists a type $C$ deformations of $V_k$ to $\ms{O}$ (\eg, $\chi_p \oplus 1$).
This shows that $p \ne 0$ in $R^{\Box, C}$.  Flatness over $\ms{O}$ now
follows from (1).  We now compute the dimension of $R^{\Box, C}$.  The
relative dimension of $R^{\Box, C}$ over
$\ms{O}$ is the same as the dimension of the tangent space of
$\ms{L}^{\Box, C}[1/p]$ at any point.  Let $\xi$ be the $E$-point of
$\ms{L}^{\Box, C}[1/p]$ corresponding to the representation $\chi_p \oplus 1$.
The dimension of the tangent space at $\xi$ is 2 (the contribution of the
framing) plus the dimension of $H^1(G_{F_v}, E(\chi_p))$ (the contribution of
deforming the representation), which gives a total of $[F_v:\Q_p]+3$.

\subsection
Let $F_v/\Q_{\ell}$ be a finite extension with $\ell \ne p$ and let $\rho_v:
G_{F_v} \to \GL_2(E)$ be a representation with $E$ any extension of $\Q_p$.
Let $V$ be the representation space of $\rho_v$.
\begin{itemize}
\item We say that $\rho_v$ is \emph{admissible of type $AB$} if it is
unramified and has determinant $\chi_p$.
\item We say that $\rho_v$ is \emph{admissible of type $C$} if there is
an exact sequence
\begin{displaymath}
0 \to E(\chi_p) \to V \to E \to 0.
\end{displaymath}
Of course, this forces $\rho_v$ to have determinant $\chi_p$.
\end{itemize}
We say that $\rho_v$ has \emph{type $\ast$} if it is admissible of type $\ast$
over some finite extension $F_v'/F_v$.  As long as $\det{\rho_v}$ is a finite
order character times the cyclotomic character, $\rho_v$ will have type $AB$
or $C$.  As in the previous case, it is possible for $\rho_v$ to have type
$AB$ and $C$; we then say that $\rho_v$ has type $AB/C$.

\subsection
We will need some results on deformations of given types, as in the case
where $v \mid p$.  We assume that $F_v$ contains the $p$th roots of unity.

\begin{proposition}
\label{prop:lring}
Let $\ol{\rho}_v:G_{F_v} \to \GL_2(k)$ be the trivial representation and
let $R^{\Box}$ denote its universal framed deformation ring.  Let $\ast$ be
one of $AB$ or $C$.  There is then a unique quotient $R^{\Box, \ast}$ of
$R^{\Box}$ satisfying the following properties:
\begin{enumerate}
\item $R^{\Box, \ast} \otimes_{\ms{O}} \ms{O}_{E'}$ is a domain, for any
finite $E'/E$.
\item $R^{\Box, \ast}$ is flat over $\ms{O}$ of relative dimension 3.
\item $R^{\Box, \ast}[1/p]$ is formally smooth over $E$.
\item Let $E'/E$ be a finite extension.  A map $R^{\Box} \to E'$
factors through $R^{\Box, \ast}$ if and only if the
corresponding representation is admissible of type $\ast$.
\end{enumerate}
\end{proposition}

\begin{proof}
For $\ast=AB$ this is easy; for $\ast=C$ it is \cite[Corollary~2.6.7]{Kisin}
with $\gamma=1$.
\end{proof}

\subsection\label{ss:type}
We now introduce some terminology for dealing with types that we will use
for the rest of the paper.
By a \emph{$p$-type} we mean one of the symbols $A$, $B$, $C$ or $A/C$.
By a \emph{prime-to-$p$-type} we mean one of the symbols $AB$, $C$ or
$AB/C$.  The types $A/C$ and $AB/C$ are \emph{indefinite}; the others
are \emph{definite}.  For two types $X$ and $Y$ we write $X \approx Y$
if $X$ and $Y$ are equal or if $(X, Y)$ is one of $(A, A/C)$, $(C, A/C)$,
$(AB, AB/C)$ or $(C, AB/C)$ up to switching order.  For example, $A
\approx A/C$ but $B \not \approx A/C$.  Note that if $X$ and $Y$ are definite
types then $X \approx Y$ is equivalent to $X=Y$.

Let $F_v/\Q_p$ be a finite extension and $\rho_v:G_{F_v} \to \GL_2(E)$ a
weight two representation.  We define $t_{\rho_v}$ to the appropriate
$p$-type, \eg, $t_{\rho_v}=A/C$ if $\rho_v$ has type $A$ and type $C$
while $t_{\rho_v}=A$ if $\rho_v$ has type $A$ and not type $C$.  We
similarly define $t_{\rho_v}$ when $F_v$ is a finite extension of $\Q_{\ell}$
with $\ell \ne p$.  We use the phrase ``$\rho_v$ has type $\ast$'' to
mean $t_{\rho_v} \approx \ast$ (which is how the phrase has been defined in
the previous sections).  For a definite type $\ast$, we use the phrase
``$\rho_v$ has definite type $\ast$'' to mean $t_{\rho_v}=\ast$.

Let $F$ be a number field and $S$ a finite set of primes of $F$.  By a
\emph{type function} over $F$ on $S$ we mean a function $t$ on $S$ such
that $t(v)$ is a $p$-type (resp.\ prime-to-$p$-type) for $v$ above $p$
(resp.\ not above $p$).  We call $t$ \emph{definite} if $t(v)$ is definite
for all $v \in S$.  If $t$ and $t'$ are two type functions over $F$ on $S$
we write $t \approx t'$ if $t(v) \approx t'(v)$ for all $v \in S$.
For a type function $t$ over $F$ on $S$ and an extension $F'/F$ we let
$t \vert_{F'}$ be the type function over $F'$ on the
primes over $S$ which assigns to a place $v'$ the value $t(v)$ where $v$ is
the place of $F$ over which $v'$ lies.

For a weight two representation $\rho:G_{F, S} \to \GL_2(E)$, with $E$ an
extension of $\Q_p$, we let $t_{\rho}$ be the type function on $S$ which
assigns to $v$ the type $t_{\rho \vert_{G_{F_v}}}$.  We call
$t_{\rho}$ the \emph{type} of $\rho$.  We have the following result:

\begin{proposition}
\label{prop:type}
Let $\rho:G_F \to \GL_2(\ol{\Q}_p)$ be an irreducible, odd, weight two
representation which is potentially modular.  Then $t_{\rho}$ is definite.
\end{proposition}

\begin{proof}
Assume that $v \nmid p$ is a finite place for which $t_{\rho}(v)$ is not
definite.  We can then find a finite totally real extension $F'/F$ and a place
$w$ of $F'$ over $v$ such that (1) $\rho \vert_{G_{F'}}=\rho_f$ for a cuspidal
Hilbert eigenform $f$ over $F'$; and (2) $\rho \vert_{G_{F'_w}}$ is unramified
and an extension of 1 by $\chi_p$.  Statement (2) shows that
$\tr{\rho(\Frob_w)}=\bN{w}+1$, which contradicts the Ramanujan conjecture
for $f$ (see \cite{Blasius}).  Thus no such $v$ exists.  To show that
$t_{\rho}(v)$ is definite for $v \mid p$ one can apply a similar argument to
crystalline Frobenius.
\end{proof}

\section{An $R=\T$ theorem}
\label{s:reqt}

\subsection\label{ss:setup}
Let $F$ be a totally real number field.  Fix a representation
\begin{displaymath}
\ol{\rho}:G_F \to \GL_2(k).
\end{displaymath}
(Recall that $k$ is the residue field of $E$, a finite extension of $\Q_p$.)
We assume the following:
\begin{enumerate}
\item[(A1)] The representation $\ol{\rho} \vert_{G_{F(\zeta_p)}}$ is
absolutely irreducible.
\item[(A2)] If $p=5$ and the projective image of $\ol{\rho}$ is $\PGL_2(\F_5)$
then $[F(\zeta_p):F]=4$.
\item[(A3)] The representation $\ol{\rho}$ is odd.
\item[(A4)] The representation $\ol{\rho}$ is everywhere unramified.
\item[(A5)] We have $\det{\ol{\rho}}=\ol{\chi}_p$.
\item[(A6)] The field $F$ has even degree over $\Q$.
\item[(A7)] The field $k$ contains the eigenvalues of the image of $\ol{\rho}$.
\end{enumerate}
Hypothesis (A6) is used in \S \ref{ss:s2def} to ensure the existence of a
certain quaternion algebra.  Hypothesis (A7) is mild:  we can always replace
$k$ by its quadratic extension to ensure that it holds.

\subsection\label{ss:defdata}
We define a \emph{deformation datum} to be a tuple $\mc{D}^{\circ}=(t,
\Sigma^{\ram}, \Sigma^{\aux})$ where:
\begin{itemize}
\item $\Sigma^{\ram}$ is a finite set of primes disjoint from $\Sigma_p$.
\item $t$ is a definite type function over $F$ on $\Sigma_p \cup
\Sigma^{\ram}$.
\item $\Sigma^{\aux}$ is a finite set of primes disjoint from $\Sigma_p
\cup \Sigma^{\ram}$ such that each $v \in \Sigma^{\aux}$ satisfies
$\bN{v} \ne 1 \pmod{p}$ and
$\tr{\ol{\rho}(\Frob_v)} \ne \pm (1+\bN{v})$.
\end{itemize}
We put $\Sigma_{\ast}=t^{-1}(\ast)$, $\Sigma_{p, \ast}=\Sigma_p \cap
\Sigma_{\ast}$ and $\Sigma^{\ram}_{\ast}=\Sigma^{\ram} \cap \Sigma_{\ast}$.
Of particular importance will be the set $\Sigma_C=\Sigma_{p, C} \cup
\Sigma^{\ram}_C$.

We fix for the rest of \S \ref{s:reqt} a deformation datum $\mc{D}^{\circ}$.
We impose the following hypotheses:
\begin{enumerate}
\item[(A8)] The representation $\ol{\rho}$ is trivial at all places above
$\Sigma_p \cup \Sigma^{\ram}$.
\item[(A9)] The set $\Sigma_C$ has even cardinality.
\end{enumerate}
The condition (A9) is used to ensure the existence of the quaternion algebra
of \S \ref{ss:s2def}.

We define a \emph{TW-extension} (of the datum $\mc{D}^{\circ}$) to be a tuple
$\mc{D}=(\Sigma^{\TW}, \{\alpha_v\})$ where:
\begin{itemize}
\item $\Sigma^{\TW}$ is a finite set of primes disjoint from $\Sigma_p
\cup \Sigma^{\ram} \cup \Sigma^{\aux}$ such that for each
$v \in \Sigma^{\aux}$ we have $\bN{v}=1 \pmod{p}$ and
the eigenvalues of $\ol{\rho}(\Frob_v)$ are distinct and belong to $k$.
\item $\alpha_v$ is an eigenvalue of $\ol{\rho}(\Frob_v)$, for $v
\in \Sigma^{\TW}$.
\end{itemize}
The \emph{trivial TW-extension} is the one where $\Sigma^{\TW}=\emptyset$.
We often write $\mc{D}^{\circ}$ for the trivial TW-extension of
$\mc{D}^{\circ}$.

\begin{remark}
We are eventually trying to prove that certain lifts $\rho$ of $\ol{\rho}$
are modular, given that $\ol{\rho}$ is modular.  In such applications,
$\Sigma^{\ram}=\Sigma^{\ram}_C$ will be the set of primes away from $p$ at
which $\rho$
ramifies, $\Sigma^{\aux}$ will be a set containing a single prime (used to
ensure that the subgroup $U^{\circ}$ of \S \ref{ss:modhypo} satisfies
the condition $(\ast)$ of \S \ref{ss:s2def}) and $\Sigma^{\TW}$ will take on
varying values (these are the primes used in the Taylor-Wiles method).
\end{remark}

\subsection\label{ss:defrings}
Let $\mc{D}$ be a TW-extension of $\mc{D}^{\circ}$.  We define several local
framed deformation rings:
\begin{itemize}
\item We let $R^{\Box}_v$ be the universal framed deformation ring of
$\ol{\rho} \vert_{G_{F_v}}$.
\item For $v \in \Sigma_p \cup \Sigma^{\ram}$ we let $R^{\Box}_{\mc{D}, v}$ be
the ``admissible of type $t(v)$'' quotient of $R^{\Box}_v$ given by
Proposition~\ref{prop:pring} in case $v \in \Sigma_p$
Proposition~\ref{prop:lring} in case $v \in \Sigma^{\ram}$.
\item For $v \in \Sigma_{p, A} \cup \Sigma_{p, B}$ we also have a ring
$\wt{R}^{\Box}_{\mc{D}, v}$ (see \cite[\S 3.4.7]{Kisin}).  This ring
enjoys the ring theoretic properties listed in Proposition~\ref{prop:pring}
and admits a finite map from $R^{\Box}_{\mc{D}, v}$ which becomes an
isomorphism after inverting $p$; other than these facts, the details of
$\wt{R}^{\Box}_{\mc{D}, v}$ will not be important to us.
\end{itemize}
We also need some semi-local framed deformation rings.  In what follows,
``tensor product'' means the completed tensor product.
\begin{itemize}
\item We let $R^{\Box, \loc}$ be the tensor product of the rings
$R^{\Box}_v$ for $v \in \Sigma_p \cup \Sigma^{\ram}$.
\item We let $R^{\Box, \loc}_{\mc{D}}$ be the tensor product of the
rings $R^{\Box}_{\mc{D}, v}$ for $v \in \Sigma_p \cup \Sigma^{\ram}$.
\item We let $\wt{R}^{\Box, \loc}_{\mc{D}}$ be the tensor product of
the $\wt{R}^{\Box}_{\mc{D}, v}$ for $v \in \Sigma_{p, A} \cup \Sigma_{p, B}$
together with the $R^{\Box}_{\mc{D}, v}$ for $v \in \Sigma_{p, C} \cup
\Sigma^{\ram}$.
\end{itemize}
The above rings depend only on $\mc{D}^{\circ}$ and not the TW-extension
$\mc{D}$.
Let $\Sigma(\mc{D})=\Sigma_p \cup \Sigma^{\ram} \cup \Sigma^{\aux} \cup
\Sigma^{\TW}$.  We now define some global framed deformation rings.
\begin{itemize}
\item Let $R^{\Box}_{\Sigma(\mc{D})}$ be the universal ring classifying
deformations of $\ol{\rho}$ unramified outside of $\Sigma(\mc{D})$, with
determinant $\chi_p$ and with framings at each place in $\Sigma_p \cup
\Sigma^{\ram}$.
\item Let $R^{\Box}_{\mc{D}}$ be the tensor product
$R^{\Box, \loc}_{\mc{D}} \otimes_{R^{\Box, \loc}}
R^{\Box}_{\Sigma(\mc{D})}$.
\item Let $\wt{R}^{\Box}_{\mc{D}}$ be the tensor product
$\wt{R}^{\Box, \loc}_{\mc{D}} \otimes_{R^{\Box, \loc}}
R^{\Box}_{\Sigma(\mc{D})}
=\wt{R}^{\Box, \loc}_{\mc{D}} \otimes_{R^{\Box, \loc}_{\mc{D}}}
R^{\Box}_{\mc{D}}$.
\end{itemize}
We will make use of a few unframed deformation rings as well.  These
exist since $\ol{\rho}$ is absolutely irreducible.
\begin{itemize}
\item We let $R_{\Sigma(\mc{D})}$ be the universal ring classifying
deformations of $\ol{\rho}$ unramified outside of $\Sigma(\mc{D})$ with
determinant $\chi_p$.
\item We let $R_{\mc{D}}$ be the quotient of $R_{\Sigma(\mc{D})}$
analogous to the quotient $R^{\Box}_{\mc{D}}$ of $R^{\Box}_{\Sigma(\mc{D})}$.
\item The ring $\wt{R}_{\mc{D}}$ is defined similarly.
\end{itemize}
None of these rings depend on the choice of eigenvalues $\alpha_v$ in the
TW-extension.  These will matter later, though.

\begin{remark}
Let $\mc{D}_1^{\circ}$ be a deformation datum and let $\mc{D}_2^{\circ}$ be
the deformation datum gotten by deleting $\Sigma^{\ram}_{AB}$ from
$\Sigma^{\ram}$.  Then $R_{\mc{D}_1^{\circ}}=R_{\mc{D}_2^{\circ}}$, as
$R_{\Sigma(\mc{D}_1^{\circ})}$ allows ramification at $\Sigma^{\ram}_{AB}$
but the local condition imposed by $R^{\Box}_{\mc{D}_1^{\circ}, v}$ for $v \in
\Sigma^{\ram}_{AB}$ exactly removes this allowance.  The equality
$R^{\Box}_{\mc{D}_1^{\circ}}=R^{\Box}_{\mc{D}_2^{\circ}}$ is not true,
however, as the first ring has framings at $\Sigma^{\ram}_{AB}$ while the
second does not.
\end{remark}

\subsection
Define
\begin{displaymath}
\begin{split}
h&=h(\mc{D}^{\circ})=\dim_k H^1(G_{F, \Sigma(\mc{D}^{\circ})},
\ad^{\circ}{\ol{\rho}}(1)) \\ 
g&=g(\mc{D}^{\circ})=h-[F:\Q]+\# \Sigma_p+\# \Sigma^{\ram}-1.
\end{split}
\end{displaymath}
The purpose of this section is to recall the following result.

\begin{proposition}
\label{prop:twexists}
Given an integer $n$ we can find a set of primes $\Sigma^{\TW}$ of $F$
satisfying the following conditions:
\begin{enumerate}
\item The set $\Sigma^{\TW}$ is disjoint from $\Sigma(\mc{D}^{\circ})$.
\item We have $\bN{v}=1 \pmod{p^n}$ for each $v \in \Sigma^{\TW}$.
\item We have $\# \Sigma^{\TW}=h$.
\item The eigenvalues of $\ol{\rho}(\Frob_v)$ are distinct for each
$v \in \Sigma^{\TW}$.
\item For each $v \in \Sigma^{\TW}$ choose an eigenvalue $\alpha_v$ of
$\ol{\rho}(\Frob_v)$ and let $\mc{D}=(\Sigma^{\TW}, \{\alpha_v\})$ be
the corresponding TW-extension of $\mc{D}^{\circ}$.  Then $R^{\Box}_{\mc{D}}$
is topologically generated over $R^{\Box, \loc}_{\mc{D}}$ by $g$ elements.
\end{enumerate}
\end{proposition}

\begin{proof}
Exactly as in \cite[Proposition~3.2.5]{Kisin}.  The proof uses (A1), (A2) and
(A3) as well as the assumptions on the primes in $\Sigma^{\aux}$.
\end{proof}

\subsection
Let $\mc{D}$ be a TW-extension of $\mc{D}^{\circ}$.  Define $\Delta_{\mc{D}}$
to be the maximal $p$-power quotient of $\prod_{v \in \Sigma^{\TW}}
(\ms{O}_{F_v}/\mf{p}_v)^{\times}$.  We denote by $\mf{a}_{\mc{D}}$ the
augmentation ideal of $\ms{O}[\Delta_{\mc{D}}]$.

We now give $R_{\Sigma(\mc{D})}$ the
structure of an $\ms{O}[\Delta_{\mc{D}}]$-algebra.  Let $v$ be an element
of $\Sigma^{\TW}$.  The restriction of the universal representation
$G_F \to \GL_2(R_{\Sigma(\mc{D})})$ to the decomposition group
$G_{F_v}$ is a sum of two characters $\eta_1 \oplus \eta_2$, where
$\eta_i:G_{F_v} \to (R_{\Sigma(\mc{D})})^{\times}$ (see \cite[Lemma~2.44]
{DDT}).  Reducing $\eta_i$
modulo the maximal ideal of $R_{\Sigma(\mc{D})}$ gives an unramified
character of $G_{F_v}$ with values in $k^{\times}$ whose value on $\Frob_v$
is one of the two eigenvalues of $\ol{\rho}(\Frob_v)$.  Change the labeling
if necessary so that $\eta_1$ corresponds to the chosen eigenvalue $\alpha_v$.
By class field theory (normalized so that uniformizers correspond to
arithmetic Frobenii), $\eta_1$ gives a map $F_v^{\times} \to
(R_{\Sigma(\mc{D})})^{\times}$ whose restriction to $\ms{O}_{F_v}^{\times}$
factors through the maximal $p$-power quotient of $(\ms{O}_{F_v}/
\mf{p}_v)^{\times}$.  Taking the product of these over $v \in
\Sigma^{\TW}$ gives a group homomorphism $\Delta_{\mc{D}} \to
(R_{\Sigma(\mc{D})})^{\times}$, which gives $R_{\Sigma(\mc{D})}$ the
structure of an $\ms{O}[\Delta_{\mc{D}}]$-algebra.  Note that this gives
any $R_{\Sigma(\mc{D})}$-algebra (\eg, $R^{\Box}_{\Sigma(\mc{D})}$ or
$R^{\Box}_{\mc{D}}$) the structure of an $\ms{O}[\Delta_{\mc{D}}]$-algebra.

\begin{proposition}
\label{prop:Rdiam}
The natural map $R^{\Box}_{\mc{D}} \to R^{\Box}_{\mc{D}^{\circ}}$ is
surjective with kernel $\mf{a}_{\mc{D}} R^{\Box}_{\mc{D}}$.  The same
statement holds for the tilde rings.
\end{proposition}

\begin{proof}
The map on functors $\Hom(R^{\Box}_{\mc{D}^{\circ}}, -) \to
\Hom(R^{\Box}_{\mc{D}}, -)$ is obviously injective, and so the map on
tangent spaces is injective, and so the map the other way on cotangent spaces
is surjective.  Nakayama's lemma now gives that $R^{\Box}_{\mc{D}} \to
R^{\Box}_{\mc{D}^{\circ}}$ is surjective.  One easily sees that this map
contains $\mf{a}_{\mc{D}} R^{\Box}_{\mc{D}}$ in its kernel and that the
representation $G_F \to \GL_2(R^{\Box}_{\mc{D}}/\mf{a}_{\mc{D}}
R^{\Box}_{\mc{D}})$ is unramified at $\Sigma^{\TW}$.  It follows that
there is a map $R_{\mc{D}^{\circ}} \to R^{\Box}_{\mc{D}}/\mf{a}_{\mc{D}}
R^{\Box}_{\mc{D}}$ and this is easily seen to be the inverse of the natural
map in the other direction.
\end{proof}

\subsection\label{ss:s2def}
Let $D$ be the unique quaternion algebra over $F$ ramified at exactly the
infinite places and at $\Sigma_C$.  This exists by
(A6) and (A9).  Pick a maximal order $\ms{O}_D$ of $D$ and for each finite
place $v \not \in \Sigma_C$
pick an isomorphism $\ms{O}_{D_v} \to M_2(\ms{O}_{F_v})$.  Let $U=\prod U_v$
be a compact open subgroup of $(D \otimes_F \A_F^f)^{\times}$, where each
$U_v$ is a compact open subgroup of $\ms{O}_{D_v}^{\times}$.  We call such
a compact open subgroup \emph{standard}.  We let $S_2(U)$ denote the
set of functions
\begin{displaymath}
f:D^{\times} \bs (D \otimes_F \A_F^f)^{\times}/((\A_F^f)^{\times} \cdot U)
\to \ms{O}.
\end{displaymath}
If $v$ is a place at which $U$ is maximal and $D$ is unramified then the Hecke
operators $T_v$ acts on $S_2(U)$.  We let $\T(U)$ (resp.\ $\T^{(p)}(U)$) be the
$\ms{O}$-subalgebra of $\End(S_2(U))$ generated by the $T_v$
for all such $v$ (resp.\ for all such $v$ not above $p$).

We will often need to impose the following condition on our subgroup $U$:
\begin{displaymath}
\textrm{For all $t \in (D \otimes_F \A_F^f)^{\times}$ the group
$((\A_F^f)^{\times} \cdot U \cap t^{-1} D^{\times} t)/F^{\times}$ has
prime-to-$p$ order.}
\eqno{(\ast)}
\end{displaymath}
This condition is equivalent to the following one:
\begin{displaymath}
\textrm{The stabilizers of $(\A_F^f)^{\times} \cdot U$ acting on $D^{\times}
\bs (D \otimes_F \A_F^f)^{\times}$ have prime-to-$p$ order.}
\end{displaymath}
Obviously, if this condition holds for $U$ then it holds for any subgroup
of $U$.

\subsection\label{ss:modhypo}
We now make the following hypothesis on the representation $\ol{\rho}$.
We assume that
there exists a standard compact open subgroup $U^{\circ}$ of $(D \otimes_F
\A_F^f)^{\times}$ and a maximal ideal $\mf{m}$ of $\T(U^{\circ})$ satisfying
the following conditions:
\begin{enumerate}
\item[(B1)] The group $U^{\circ}_v$ is maximal except for $v \in \Sigma^{\aux}$.
\item[(B2)] The group $U^{\circ}$ satisfies the condition $(\ast)$.
\item[(B3)] For $v \not \in \Sigma(\mc{D}^{\circ})$
the image of $T_v$ in $\T(U^{\circ})/\mf{m}$ is equal to
$\tr{\ol{\rho}(\Frob_v)}$.
\item[(B4)] For $v \in \Sigma_{p, A}$ we have $T_v \not \in \mf{m}$.
\item[(B5)] For $v \in \Sigma_{p, B}$ we have $T_v \in \mf{m}$.
\item[(B6)] The residue field of $\mf{m}$ is $k$.
\end{enumerate}
Condition (B6) is mild and can always be ensured by enlarging $k$.
The ideal $\mf{m}$ is forced to be non-Eisenstein since the representation
$\ol{\rho}$ is absolutely irreducible.

\begin{remark}
Let $U$ be a standard compact open with $U_v$ maximal for some place $v \in
\Sigma_p \setminus \Sigma_{p, C}$.
Let $f$ be an element of $S_2(U)$ which is an eigenform for $\T(U)$.
Associated to $f$ is a maximal ideal $\mf{m}$ of $\T(U)$.  One then knows
that the $p$-adic Galois representation $\rho_f$ associated to $f$ is
admissible of type
$A$ (resp.\ $B$) at $v$ if and only if $T_v \not \in \mf{m}$ (resp.\ $T_v
\in \mf{m}$).  This follows from the second part of \cite[Lemma~3.4.2]{Kisin}.
This is the reason for conditions (B4) and (B5) above.
\end{remark}

\subsection
Let $\mc{D}$ be a TW-extension of $\mc{D}^{\circ}$.
Define standard compact open subgroups $U_{\mc{D}}$ and $U^-_{\mc{D}}$ by
$(U_{\mc{D}})_v=(U^-_{\mc{D}})_v=U^{\circ}_v$ for $v \not \in \Sigma^{\TW}$
and
\begin{displaymath}
(U^-_{\mc{D}})_v=\left\{ \mat{a}{b}{c}{d} \in \GL_2(\ms{O}_{F_v})
\;\bigg\vert\; \textrm{$c \in \mf{p}_v$} \right\}
\end{displaymath}
and
\begin{displaymath}
(U_{\mc{D}})_v=\left\{ \mat{a}{b}{c}{d} \in \GL_2(\ms{O}_{F_v})
\;\bigg\vert\; \textrm{$c \in \mf{p}_v$ and $ad^{-1}$ maps to 1 in
$\Delta_{\mc{D}}$} \right\}
\end{displaymath}
for $v \in \Sigma^{\TW}$.  We identify $U_{\mc{D}}^-/U_{\mc{D}}$ with
the group $\Delta_{\mc{D}}$ by locally mapping a matrix to $ad^{-1}$.
As such, the group $\Delta_{\mc{D}}$ naturally acts on the space
$S_2(U_{\mc{D}})$.

For any four of the Hecke algebras $\T(U_{\mc{D}})$, $\T^{(p)}(U_{\mc{D}})$,
$\T(U_{\mc{D}}^-)$ or $\T^{(p)}(U_{\mc{D}}^-)$ we put a $+$ in the superscript
to indicate the Hecke algebra generated by the given algebra and the
Atkin-Lehner operators $U_v$ for $v \in \Sigma^{\TW}$.  We regard
$\mf{m}$ as an ideal of each of these algebras.

\begin{proposition}
For each $v \in \Sigma^{\TW}$ choose an eigenvalue $\beta_v$ of
$\ol{\rho}(\Frob_v)$.  Then the ideal of $\T^+(U_{\mc{D}}^-)$ generated by
$\mf{m}$ and the $U_v-\beta_v$ for $v \in \Sigma^{\TW}$ is a maximal ideal.
Every maximal ideal above $\mf{m}$ is of this form and no two are the same.
\end{proposition}

\begin{proof}
This follows from statement (2) of \cite[Lemma~1.6]{Taylor} and statement (1)
of \cite[Corollary~1.8]{Taylor}
\end{proof}

We let $\mf{m}^-_{\mc{D}}$ be the maximal ideal of $\T^+(U_{\mc{D}}^-)$
corresponding to the eigenvalues $\{\alpha_v\}$ given in the datum
$\mc{D}$ and let $\mf{m}_{\mc{D}}$ be the induced maximal ideal of
$\T^+(U_{\mc{D}})$.
We then define $\T_{\mc{D}}$ (resp.\ $\T^{(p)}_{\mc{D}}$) to
be the localization $\T^+(U_{\mc{D}})_{\mf{m}_{\mc{D}}}$
(resp.\ $\T^{(p), +}(U_{\mc{D}})_{\mf{m}_{\mc{D}}}$).
We also put $M_{\mc{D}}=S_2(U_{\mc{D}})_{\mf{m}_{\mc{D}}}$, which is a
module over $\T_{\mc{D}}$.
Note that under
our convention of identifying $\mc{D}^{\circ}$ with its trivial TW-extension
we have $\T_{\mc{D}^{\circ}}=\T(U^{\circ})_{\mf{m}}$, and similarly for
the prime-to-$p$ version.

\begin{proposition}
\label{prop:Tdiam}
The space $M_{\mc{D}}$ is a free $\ms{O}[\Delta_{\mc{D}}]$-module and
there is a natural isomorphism $M_{\mc{D}}/\mf{a}_{\mc{D}} M_{\mc{D}} \to
M_{\mc{D}^{\circ}}$.
\end{proposition}

\begin{proof}
This is proved as \cite[Corollary~2.4]{Taylor}.
\end{proof}

\subsection
Fix a TW-extension $\mc{D}$ of $\mc{D}^{\circ}$.  We then have the following:

\begin{proposition}
There is a representation $\rho:G_F \to \GL_2(\T^{(p)}_{\mc{D}})$ which is
unramified outside of $\Sigma(\mc{D})$, has determinant $\chi_p$ and
satisfies $\tr{\rho}(\Frob_v)=T_v$ for $v \not \in \Sigma(\mc{D})$.  The
corresponding map $R_{\Sigma(\mc{D})} \to \T^{(p)}_{\mc{D}}$ is surjective.
\end{proposition}

\begin{proof}
This follows from \cite{Taylor2} and the Jacquet-Langlands correspondence.
\end{proof}

We put $\T^{\Box}_{\mc{D}}=R^{\Box}_{\Sigma(\mc{D})}
\otimes_{R_{\Sigma(\mc{D})}}
\T_{\mc{D}}$ and similarly for the prime-to-$p$ version.  We also define
$M^{\Box}_{\mc{D}}=\T^{\Box}_{\mc{D}} \otimes_{\T_{\mc{D}}} M_{\mc{D}}$.
We regard $M_{\mc{D}}$ as an $R_{\Sigma(\mc{D})}$-module via the map
$R_{\Sigma(\mc{D})} \to \T^{(p)}_{\mc{D}}$, and similarly for the boxed
version.  We note that the two actions of $\Delta_{\mc{D}}$ on $M_{\mc{D}}$
--- one coming from the $R_{\Sigma(\mc{D})}$-module structure, the other from
the identification $U_{\mc{D}}^-/U_{\mc{D}}=\Delta_{\mc{D}}$ --- agree (this
follows from \cite[Corollary~1.8]{Taylor}).

\begin{proposition}
The composite map
\begin{displaymath}
R^{\Box, \loc} \to R^{\Box}_{\Sigma(\mc{D})} \to \T^{\Box, (p)}_{\mc{D}}
\to \T^{\Box}_{\mc{D}}
\end{displaymath}
factors through $R^{\Box, \loc}_{\mc{D}}$.  The resulting map
$R^{\Box}_{\mc{D}} \to \T^{\Box}_{\mc{D}}$ extends naturally to a surjection
\begin{displaymath}
\wt{R}^{\Box}_{\mc{D}} \to \T^{\Box}_{\mc{D}}.
\end{displaymath}
\end{proposition}

\begin{proof}
For the first statement, it suffices to show that for each $v \in
\Sigma_p \cup \Sigma^{\ram}$ the map $R^{\Box}_v \to \T^{\Box}_{\mc{D}}$
factors through $R^{\Box}_{\mc{D}, v}$.  For $v \in \Sigma_{p, A} \cup
\Sigma_{p, B} \cup \Sigma^{\ram}_C$ this is shown in \cite[Lemma~3.4.9]{Kisin}.
The proof for $v \in \Sigma^{\ram}_{AB}$ is the same as the $v \in
\Sigma^{\ram}_C$ case.
We must handle the $v \in \Sigma_{p, C}$ case.  As in \cite[Lemma~3.4.9]{Kisin}
it suffices to show that for any map $\T^{\Box}_{\mc{D}} \to E'$, with
$E'$ an extension of $E$, the resulting map $R^{\Box}_v \to E'$ factors through
$R^{\Box}_{\mc{D}, v}$.  The map $\T^{\Box}_{\mc{D}} \to E'$ determines
a Hilbert modular form $f$ which is special of conductor one at $v$.  We thus
have
\begin{displaymath}
\rho_f \vert_{G_{F_v}}=\mat{\chi_p \eta}{\ast}{}{\eta}
\end{displaymath}
where $\eta$ is unramified and $\eta^2=1$.  Since $\rho_f \vert_{G_{F_v}}$
reduces to the trivial representation (and $p \ne 2$) we conclude $\eta=1$.
Thus $\rho_f$ is admissible of type $C$ at $v$.  This
proves that $R^{\Box}_v \to E'$ factors through $R^{\Box}_{\mc{D}, v}$.
That $R^{\Box}_{\mc{D}} \to \T^{\Box}_{\mc{D}}$ extends to a surjection
from $\wt{R}^{\Box}_{\mc{D}}$ is a local statement at the primes in
$\Sigma_{p, A} \cup \Sigma_{p, B}$ and is proved in \cite[Lemma~3.4.9]{Kisin}.
\end{proof}

\subsection
We recall the following result from \cite[Proposition~3.3.1]{Kisin}.

\begin{proposition}
\label{prop:patch}
Let $B$ be a complete, local, flat $\ms{O}$-algebra, which is a domain of
dimension $d+1$ and such that $B[1/p]$ is formally smooth over $E$.  Let
$R$ be a local $B$-algebra and $M$ a non-zero $R$-module.  Suppose that
there are integers $h$ and $j$ such that for any integer $n$ we can find:
\begin{itemize}
\item a local $B$-algebra $R'$, topologically generated over $B$ by $h+j-d$
elements;
\item a map $\ms{O} \lbb x_1, \ldots, x_h, y_1, \ldots, y_j \rbb
\to R'$;
\item a surjection of $B$-algebras $R' \to R$ with kernel $(x_1, \ldots,
x_h) R'$;
\item an $R'$-module $M'$; and
\item a surjection of $R'$-modules $M' \to M$ with kernel $(x_1, \ldots,
x_h) M'$,
\end{itemize}
such that the following condition holds:
\begin{itemize}
\item if $\mf{b}' \subset \ms{O} \lbb x_1, \ldots, x_h, y_1, \ldots, y_j
\rbb$ is the annihilator of $M'$ then
\begin{displaymath}
\mf{b}' \subset ((1+x_1)^{p^n}-1, \ldots, (1+x_h)^{p^n}-1)
\end{displaymath}
and $M'$ is finite free over $\ms{O} \lbb x_1, \ldots, x_h, y_1, \ldots, y_j
\rbb/\mf{b}'$.
\end{itemize}
Then $R$ is a finite $\ms{O} \lbb y_1, \ldots, y_j \rbb$-algebra and
$M[1/p]$ is a finite, projective and faithful $R[1/p]$-module.
\end{proposition}

\subsection
The aim of this section is to prove the following theorem.  We follow
\cite[Theorem~3.4.11]{Kisin}.

\begin{theorem}
\label{thm:ReqT}
The map $\wt{R}^{\Box}_{\mc{D}^{\circ}} \to \T^{\Box}_{\mc{D}^{\circ}}$ is
surjective and has $p$-power torsion kernel.  The ring $R_{\mc{D}^{\circ}}$ is
finite over $\ms{O}$.
\end{theorem}

\begin{proof}
We apply Proposition~\ref{prop:patch} with:
\begin{displaymath}
B=\wt{R}^{\Box, \loc}_{\mc{D}^{\circ}}, \qquad
R=\wt{R}^{\Box}_{\mc{D}^{\circ}}, \qquad
M=M_{\mc{D}^{\circ}}^{\Box}, \qquad
h=h(\mc{D}^{\circ}), \qquad
j=4(\# \Sigma_p+\# \Sigma^{\ram})-1.
\end{displaymath}
By Proposition~\ref{prop:pring}, Proposition~\ref{prop:lring} and
\cite[Lemma~3.4.12]{Kisin} $B$ is a complete, local, flat $\ms{O}$-algebra
which is a domain and such that $B[1/p]$ is smooth over $E$.  Its dimension
is $d+1$ where $d$ is the relative dimension of $R^{\Box, \loc}_{\mc{D}}$
over $\ms{O}$.  Thus
\begin{displaymath}
d=\sum_{v \in \Sigma_p} ([F_v:\Q_p]+3) + \sum_{v \in \Sigma^{\ram}} 3
=[F:\Q]+3 \# \Sigma_p+3 \# \Sigma^{\ram}
\end{displaymath}
Let $n$ be a given integer, let $\mc{D}$ be the TW-extension produced by
Proposition~\ref{prop:twexists} and put $R'=\wt{R}^{\Box}_{\mc{D}}$ and
$M'=M^{\Box}_{\mc{D}}$.  The morphism $R^{\Box}_{\Sigma(\mc{D})} \to
R_{\Sigma(\mc{D})}$ is formally smooth of relative dimension $j$, and so we can
identify $R^{\Box}_{\Sigma(\mc{D})}$ with $R_{\Sigma(\mc{D})} \lbb y_1, \ldots
y_j \rbb$.  In particular, this gives $R'$ the structure of an
$\ms{O} \lbb y_1, \ldots, y_j \rbb$-algebra.  The group $\Delta_{\mc{D}}$
is a quotient of the product of $h$ cyclic groups.  We can thus write
$\ms{O}[\Delta_{\mc{D}}]$ as a quotient of $\ms{O} \lbb x_1, \ldots, x_h
\rbb$ in such a way that the images of the $x_i$ generate the augmentation
ideal $\mf{a}_{\mc{D}}$.  As $R'$ is an algebra over
$\ms{O}[\Delta_{\mc{D}}]$ we can regard it as an algebra over $\ms{O} \lbb
x_1, \ldots, x_h \rbb$.  We have thus defined a map $\ms{O} \lbb x_1, \ldots,
x_h, y_1, \ldots, y_j \rbb \to R'$.

Proposition~\ref{prop:Rdiam} shows that the natural map $R' \to R$ is
surjective and has kernel $(x_1, \ldots, x_h)R'$.  Similarly,
Proposition~\ref{prop:Tdiam} shows that the natural map $M' \to M$ is
surjective and has kernel $(x_1, \ldots, x_h)M'$.  That proposition, together
with obvious properties of $\Delta_{\mc{D}}$, show that $\mf{b}'$ satisfies
the necessary condition and that $M'$ is finite and
free over $\ms{O} \lbb x_1, \ldots, x_h, y_1, \ldots, y_j \rbb/\mf{b}'$.
Finally, we need to check that $R'$ is topologically generated as a $B$-algebra
by $h+j-d$ elements.  From the manner in which we chose $\mc{D}$ we know that
$R'$ is generated by $g=h-[F:\Q]+\# \Sigma_p+\# \Sigma^{\ram}-1$ elements
over $B$.  It suffices, therefore, to show $g \le h+j-d$.  In fact,
\begin{displaymath}
\begin{split}
h+j-d
&= h+4(\#\Sigma_p+\#\Sigma^{\ram})-1-[F:\Q]-3 \# \Sigma_p-3 \# \Sigma^{\ram} \\
&= h-[F:\Q]+\# \Sigma_p + \# \Sigma^{\ram}-1 \\
&= g.
\end{split}
\end{displaymath}
We have thus verified the hypotheses of Proposition~\ref{prop:patch}.
It follows that $M^{\Box}_{\mc{D}^{\circ}}[1/p]$ is a faithful
$\wt{R}^{\Box}_{\mc{D}^{\circ}}[1/p]$-module.  Since the
$\wt{R}^{\Box}_{\mc{D}^{\circ}}$-module structure on
$M^{\Box}_{\mc{D}^{\circ}}$ comes via the map map
$\wt{R}^{\Box}_{\mc{D}^{\circ}} \to \T^{\Box}_{\mc{D}^{\circ}}$, it follows
that this map must have $p$-power torsion kernel, proving the first statement
of the theorem.  Proposition~\ref{prop:patch} also implies that
$\wt{R}^{\Box}_{\mc{D}^{\circ}}$ is finite over $\ms{O} \lbb y_1, \ldots, y_j
\rbb$.  Since $\wt{R}^{\Box}_{\mc{D}^{\circ}}$ is identified with
$\wt{R}_{\mc{D}^{\circ}} \lbb y_1, \ldots, y_j \rbb$ it follows that
$\wt{R}_{\mc{D}^{\circ}}$ is finite over $\ms{O}$.  There does not seem to
be an obvious way to deduce the finiteness of $R_{\mc{D}^{\circ}}$ from
that of $\wt{R}_{\mc{D}^{\circ}}$; however, one can run the entirety of
the above argument using the non-tilde rings to conclude that
$R_{\mc{D}^{\circ}}$ is finite over $\ms{O}$.
\end{proof}

\section{Modularity lifting theorems}
\label{s:modlift}

\subsection
We now give our first modularity lifting theorem.  By a ``Hilbert eigenform
with coefficients in $\ol{\Q}_p$'' we mean a Hilbert eigenform together with
an embedding of its coefficient field into $\ol{\Q}_p$.  For such a form $f$,
unramified outside a finite set of primes $S$, we write $t_f$ for the type
function on $S$ of its associated Galois representation
$\rho_f$ (see \S \ref{ss:type}).  Note that $t_f$ is definite by
Proposition~\ref{prop:type}.

\begin{theorem}
\label{thm:mlt2}
Let $F$ be a totally real number field, $S$ a finite set of primes and
$\rho:G_{F,S} \to \GL_2(\ol{\Q}_p)$ an odd, weight two representation such
that $\ol{\rho}$ satisfies (A1) and (A2) of \S \ref{s:reqt}.  Assume there
exists a parallel weight two Hilbert eigenform $f$ with coefficients in
$\ol{\Q}_p$ such that $f$ is unramified outside $S$, $\ol{\rho} \cong
\ol{\rho}_f$ and $t_{\rho} \approx t_f$.  Then there is a Hilbert eigenform
$g$ with coefficients in $\ol{\Q}_p$ such that $\rho \cong \rho_g$.
\end{theorem}

\begin{proof}
We call an extension $F'/F$ \emph{pre-solvable} if its Galois closure is
solvable.  This property behaves well in towers (unlike the property of
being solvable, which requires Galois) and under compositum.  By Langlands'
base change, it suffices to prove that $\rho$ is modular after a pre-solvable
extension.  To begin with, we can pass to a totally real pre-solvable
extension so that the following condition holds:
\begin{itemize}
\item $\det{\rho}$ and $\det{\rho_f}$ are of the form $\psi^2 \chi_p$
(with possibly different $\psi$).
\end{itemize}
(This is proved in Corollary~\ref{cor:sqr} below.)
We can thus replace $\rho$ and $\rho_f$ with twists so that they have
determinant $\chi_p$.  Pick a finite extension $E/\Q_p$ such that $\rho$ and
$\rho_f$ are defined over $E$.  After making another totally real pre-solvable
base change, and possibly enlarging $E$, we can now ensure the following:
\begin{itemize}
\item $\ol{\rho}$ still satisfies (A1) and (A2) from \S \ref{s:reqt}.
\item $\ol{\rho}$ is everywhere unramified.
\item $\rho$ and $\rho_f$ are admissible at all finite places.
\item $\ol{\rho}$ is trivial at all places where $\rho$ ramifies and all
places above $p$.
\item The set of primes at which $\rho_f$ have type $C$ has even cardinality.
\item $F$ has even degree over $\Q$.
\item $k$ contains the eigenvalues of $\ol{\rho}$.
\end{itemize}
We thus see that $\ol{\rho}$ satisfies (A1)--(A7) of \S \ref{ss:setup}.  Note
that if $v \nmid p$ is a place at which $\rho$ ramifies then the image of
inertia at $v$ under $\rho$ and $\rho_f$ is unipotent and $\pi_{f, v}$ is
special of conductor one.

We now
define a deformation datum $\mc{D}^{\circ}=(t, \Sigma^{\ram}, \Sigma^{\aux})$.
We let $\Sigma^{\ram}$ be the set of primes away from $p$ at which $\rho_f$
ramifies.  We let $t=t_f$.  We let
$\Sigma^{\aux}$ be the set $\{w\}$ where $w$ is any large place of $F$ not in
$\Sigma_p \cup \Sigma^{\ram}$ such that $\bN{w} \ne 1 \pmod{p}$ and
$\tr{\ol{\rho}}(\Frob_w) \ne \pm (1+\bN{w})$.  The existence of such a $w$ is
guaranteed by \cite[Lemma~4.11]{DDT} since $\ol{\rho}$ satisfies (A1) and the
order of $\ol{\chi}_p$ is even.  Conditions (A8) and (A9) of
\S \ref{ss:defrings} are fulfilled.

Now let $D$ be the unique quaternion algebra over $F$ ramified at the
infinite places and at $\Sigma_C$.  Let $U^{\circ}$ be
a standard compact open which is maximal everywhere except at $w$ and
satisfies $(\ast)$.  Let $\mf{m}$ be the maximal ideal of $\T(U^{\circ})$
determined by $f$ (after applying the Jacquet-Langlands correspondence).
Then $U^{\circ}$ and $\mf{m}$ satisfy conditions (B1)--(B6) of
\S \ref{ss:modhypo}, after possibly enlarging $k$.  (As remarked there,
conditions (B4) and (B5) follow from the fact $t=t_f$.)
We now apply Theorem~\ref{thm:ReqT} to conclude that $R^{\Box}_{\mc{D}^{\circ}}
\otimes_{\ms{O}} E =\T^{\Box}_{\mc{D}^{\circ}} \otimes_{\ms{O}} E$.
(Recall that $R^{\Box}_{\mc{D}^{\circ}}[1/p]=\wt{R}^{\Box}_{\mc{D}^{\circ}}
[1/p]$.) The representation $\rho$ gives a $\ol{\Q}_p$-point of
$R^{\Box}_{\mc{D}^{\circ}}$ (this is where we use $t_{\rho} \approx t_f$).
The corresponding $\ol{\Q}_p$-point of $\T^{\Box}_{\mc{D}^{\circ}}$ defines
the requisite Hilbert eigenform $g$.
\end{proof}

\subsection
We now give a proof of the elementary fact concerning characters used in the
first step of the proof of Theorem~\ref{thm:mlt2}.  We thank Bhargav
Bhatt for the proof.

\begin{lemma}
\label{lem:sqr}
Let $F$ be a number field and let $\psi:G_F \to \Q/\Z$ be a finite order
character.  Then $\psi$ is the square of another character if and only if
it is locally at all places.
\end{lemma}

\begin{proof}
We have an exact sequence of abelian groups
\begin{displaymath}
\xymatrix{
0 \ar[r] & \Z/2 \Z \ar[r] & \Q/\Z \ar[r]^2 & \Q/\Z \ar[r] & 0.}
\end{displaymath}
Regarding these as trivial Galois modules and taking cohomology gives
\begin{displaymath}
\xymatrix{
\Hom(G_{F_v}, \Q/\Z) \ar[r]^2 & \Hom(G_{F_v}, \Q/\Z) \ar[r] & \Br(F_v)[2] \\
\Hom(G_F, \Q/\Z) \ar[r]^2 \ar[u] & \Hom(G_F, \Q/\Z) \ar[r] \ar[u] &
\Br(F)[2] \ar[u] }
\end{displaymath}
where $\Br(\cdot)[2]$ is the 2-torsion of the Brauer group.  Let $\psi$ be
a character of $G_F$ and $\alpha$ the class in $\Br(F)[2]$ given by the
above map.  If $\psi$ is locally a square then $\alpha$ maps to zero in
each $\Br(F_v)[2]$.  This implies that $\alpha=0$ and so $\psi$ is a square.
\end{proof}

\begin{corollary}
\label{cor:sqr}
Let $F$ be a totally real field, let $M/F$ be a finite extension and let
$\psi:G_F \to \Q/\Z$ be a finite order character such that $\psi(c)=1$ for
each complex conjugation $c$.  Then there exists a finite, totally real,
pre-solvable extension $F'/F$, linearly disjoint from $M$, such that
$\psi \vert_{G_{F'}}=\eta^2$ where $\eta:G_{F'} \to \Q/\Z$ is an unramified
character.
\end{corollary}

\begin{proof}
Pick a finite, totally real, pre-solvable extension $F_1/F$ linearly disjoint
from $M$ such that $\psi \vert_{G_{F_1}}$ is everywhere unramified.  Then
$\psi \vert_{G_{F_1}}$ is locally a square since at any finite place it can
be regarded as a character of $\widehat{\Z}$ (and any map $\widehat{\Z} \to
\Q/\Z$ is a square) while at infinite places it is
trivial (and thus a square).  Thus $\psi \vert_{G_{F_1}}=\eta^2$.  Now pick
a finite, totally real, pre-solvable extension $F'/F_1$ linearly disjoint from
$M$ such that $\eta \vert_{G_{F'}}$ is everywhere unramified.
\end{proof}

\subsection
The hypothesis $t_{\rho} \approx t_f$ occurring in Theorem~\ref{thm:mlt2} is
more restrictive than one would like.  We now examine ways of removing or
relaxing it.  We begin by stating the following conjecture:

\begin{conjecture}
\label{conj:type}
Let $F$ be a totally real field, $f$ a cuspidal parallel weight two Hilbert
eigenform with coefficients in $\ol{\Q}_p$ unramified outside of some finite
set $S$, $t$ a definite type function on $S$ and $M/F$ a finite extension.
Then there exists a finite, solvable, totally real extension $F'/F$, linearly
disjoint from $M$, and a cuspidal parallel weight two Hilbert eigenform $f'$
over $F'$ with coefficients in $\ol{\Q}_p$ such that $t_{f'}=
t \vert_{F'}$ and $\ol{\rho}_{f'}=\ol{\rho}_f \vert_{G_{F'}}$.
\end{conjecture}

The results of \S \ref{s:problems} can be used to show that the analogous
conjecture holds on the Galois side.
The condition $t_{f'}=t \vert_{F'}$ in the conjecture is equivalent
to $t_{f'} \approx t \vert_{F'}$ as each type function is definite.
If this conjecture held then we could remove the $t_{\rho} \approx t_f$
condition from Theorem~\ref{thm:mlt2}.  Indeed, say $\ol{\rho}=\ol{\rho}_f$
is given.  By the conjecture, one can then make a solvable base change $F'/F$
so that $\ol{\rho} \vert_{G_{F'}}=\ol{\rho}_{f'}$ and $t_{\rho} \vert_{F'}
\approx t_{f'}$.  Theorem~\ref{thm:mlt2} then gives that $\rho \vert_{G_{F'}}$
is modular.  Finally, Langlands' base change gives that $\rho$ is modular.

In the following sections, we will establish certain instances of
Conjecture~\ref{conj:type}.  Each such instance will give a strengthening
of Theorem~\ref{thm:mlt2}, as outlined in the previous paragraph.

\subsection
We now establish Conjecture~\ref{conj:type} away from $p$.  Precisely,
we prove the following:

\begin{proposition}
\label{prop:lev}
Conjecture~\ref{conj:type} is true if $\ol{\rho}_f$ satisfies (A1) and
$t_f \vert_{\Sigma_p}=t \vert_{\Sigma_p}$.
\end{proposition}

In fact, \cite[\S 3.5]{Kisin} essentially contains a proof of this result.
We follow that section very closely, making only the necessary minor changes.
Before getting into the proof, we note that the proposition gives
Theorem~\ref{mainthm3}, which we restate in our present language:

\begin{theorem}
\label{thm:mlt}
Let $\rho:G_F \to \GL_2(\ol{\Q}_p)$ an odd, finitely ramified, weight two
representation such that $\ol{\rho}$ satisfies (A1) and (A2) of
\S \ref{s:reqt}.  Assume there exists a parallel weight two Hilbert eigenform
$f$ with coefficients in $\ol{\Q}_p$ such that $\ol{\rho} \cong \ol{\rho}_f$
and $t_{\rho} \vert_{\Sigma_p} \approx t_f \vert_{\Sigma_p}$.  Then there is a
Hilbert eigenform $g$ with coefficients in $\ol{\Q}_p$ such that $\rho \cong
\rho_g$.
\end{theorem}

We now begin on the proof of Proposition~\ref{prop:lev}.  We need a few
lemmas.

\begin{lemma}
\label{lem:lower}
Let $f$ be a parallel weight two Hilbert eigenform over $F$ with coefficients
in $\ol{\Q}_p$ and trivial nebentypus and let $M/F$ be a finite extension.
Assume that $\ol{\rho}_f$ satisfies (A1).  Then there exists a
finite, pre-solvable, totally real extension $F'/F$ linearly disjoint from
$M$ and a parallel weight two Hilbert modular form $f'$ over $F'$ with
coefficients in $\ol{\Q}_p$ and trivial nebentypus such that $\ol{\rho}_{f'}
\cong \ol{\rho}_f \vert_{G_{F'}}$; $t_{f'}(v)=(t_f \vert_{F'})(v)$ at places
$v$ above $p$; and $\pi_{f'}$ is special of conductor one at the primes above
$p$ for which $\rho_{f'}$ is type $C$, and unramified everywhere else.
\end{lemma}

\begin{proof}
By replacing $F$ with a finite, pre-solvable, totally real
extension we may assume that $\ol{\rho}_f$ is everywhere unramified, $F$ has
even degree over $\Q$, the set of places above $p$ at which $\rho_f$ has
type $C$ has even cardinality and $\pi_f$ is either unramified or special
of conductor one at all places.  We now explain the small modifications to
\cite{SkinnerWiles} needed to prove the lemma.  Let $D$ be the unique
quaternion algebra over $F$ ramified at the infinite places and at the places
above $p$ where $f$ has type $C$.  We use this division algebra instead of the
one used in \cite{SkinnerWiles}.  We modify the definition of the Hecke
algebra $\mathbf{T}$ used in \cite{SkinnerWiles} on p.~21 as in the proof of
\cite[Lemma~3.5.2]{Kisin} except that we only include the Hecke operators at
places above $p$ where $\rho_f$ has type $A$ or $B$.  The proof now continues
unchanged.
\end{proof}

We now describe a level-raising result we will need.  Let $D$ be a
quaternion algebra over $F$ ramified at all the infinite places and some set of
finite places.  Let $U$ be a standard maximal compact subgroup satisfying
$(\ast)$.  Let $w$ be a finite place of $F$ not above $p$, at which $D$ is
unramified and at which $U_w$ is maximal compact.
Let $U'$ be the standard compact subgroup given by $U'_v=U_v$ for $v \ne w$ and
\begin{displaymath}
U'_w=\Gamma_0(w)=\left\{ \mat{a}{b}{c}{d} \in \GL_2(\ms{O}_{F_w})
\,\bigg\vert\, c \in \mf{p}_w \right\}
\end{displaymath}
Define a map
\begin{displaymath}
i:S_2(U)^{\oplus 2} \to S_2(U'), \qquad (f_1, f_2) \mapsto
f_1+\mat{1}{}{}{\varpi_w} f_2.
\end{displaymath}
We now have the relevant result:

\begin{lemma}
\label{lem:raise}
Notation being as above, let $\mf{m}$ be a non-Eisenstein maximal ideal of
$\T(U)$.  Assume that $T_w= \pm (\bN{w}+1) \pmod{\mf{m}}$.  Then
$S_2(U')_{\mf{m}}/\im{i}$ is a non-zero free $\ms{O}$-module.
\end{lemma}

\begin{proof}
The proof goes the same as the one in \cite[\S 3.1.7, \S 3.1.9]{Kisin}.
\end{proof}

We prove one more lemma:

\begin{lemma}
\label{lem:raise2}
Let $f$ be a parallel weight two Hilbert modular form over $F$ with
coefficients in $\ol{\Q}_p$ and trivial nebentypus.  Assume that $F$ has even
degree over $\Q$,
$\ol{\rho}_f$ satisfies (A1), the set of primes over $p$ at which $\rho_f$ is
type $C$ has even cardinality and $\pi_f$ is
unramified away from $p$.  Let $M/F$ be a finite extension, let $S=\{v_1,
\ldots, v_r\}$ be a set of primes of $f$ disjoint from $\Sigma_p$ at which
$\ol{\rho}_f$ is trivial and such that $\bN{v_i}=1 \pmod{p}$ and let $w$ be
a place of $F$ not in $\Sigma_p \cup S$.
We can then find a tower of fields $F=F_0 \subset \cdots \subset F_r$ and for
each $i \in \{0, \ldots, r\}$ a parallel weight two Hilbert modular form $f_i$
over $F_i$ with coefficients in $\ol{\Q}_p$ and trivial nebentypus such that:
\begin{enumerate}
\item $F_i/F_{i-1}$ is finite, abelian and totally real.
\item $F_i$ is linearly disjoint from $M$.
\item $w$ splits completely in $F_i$.
\item There are an even number of primes in $F_i$ lying above $\{v_1, \ldots,
v_i\}$.
\item The primes $\{v_{i+1}, \ldots, v_r\}$ are inert in $F_i$.
\item There are an even number of primes in $F_i$ above $p$ at which
$\rho_{f_i}$ has type $C$.
\item $\ol{\rho}_{f_i} \cong \ol{\rho}_f \vert_{G_{F_i}}$
\item $(t_{f_i})(v)=(t_f \vert_{F_i})(v)$ at places $v$ above $p$.
\item $\pi_{f_i}$ is special of conductor one at the primes over $\{v_1,
\ldots, v_i \}$ and the primes above $p$ at which $\rho_f$ is type $C$ and
unramified at all other primes except possibly those lying over $w$.
\end{enumerate}
\end{lemma}

\begin{proof}
Replace $M$ with the compositum of $M$ and $\ker{\ol{\rho}_f}$.
We follow \cite[Lemma~3.5.3]{Kisin} and prove the lemma by induction.  Thus
assume
we have constructed $F_0 \subset \cdots \subset F_{i-1}$ and $f_0, \ldots,
f_{i-1}$.  Let $D$ be the unique quaternion algebra over $F_{i-1}$ ramified
at the infinite places, the places above $p$ at which $\rho_{f_{i-1}}$ has
type $C$ and the primes lying above $\{v_1, \ldots, v_{i-1}\}$.  Let $U$ be
a standard compact subgroup satisfying the following:  $U$ is maximal
everywhere except for the primes above $w$; $U$ satisfies $(\ast)$; $f_{i-1}$
corresponds to an eigenform in $S_2(U)$ under the Jacquet-Langlands
correspondence (which we will still denote by $f_{i-1}$).
Let $U'$ be the standard compact given by
$U'_v=U_v$ unless $v$ is the unique prime of $F_{i-1}$ lying over $v_i$,
in which case $U'_v=\Gamma_0(v)$.
Let $\mf{m}$ be the maximal ideal of $\T(U)$ determined by $f$.  Since
$\ol{\rho}_f$ is trivial at $v_i$ we have
\begin{displaymath}
\tr{\ol{\rho}_f(\Frob_{v_i})}=T_{v_i}=2 \pmod{\mf{m}}, \qquad \textrm{and}
\qquad
\det{\ol{\rho}_f(\Frob_{v_i})}=\bN{v_i}=1 \pmod{\mf{m}}.
\end{displaymath}
As $\bN{v_i}=1$ we see that the
condition of Lemma~\ref{lem:raise} is satisfied and so $S_2(U')_{\mf{m}}/
\im{i}$ is a non-zero free $\ms{O}$-module.  We can therefore pick a
non-zero eigenform $f_i'$ in $S_2(U')$ not in the image of $i$ and an
embedding of the coefficient field of $f_i'$ into $\ol{\Q}_p$ such that
$\ol{\rho}_{f_i'} \cong \ol{\rho}_{f_{i-1}}$.  Because $f_i'$ is a newform for
$U'$ it follows that $\pi_{f_i'}$ will be special at all the primes in $\{
v_1, \ldots, v_i \}$
and at the primes above $p$ at which $\rho_{f_i'}$ have type $C$ and
unramified everywhere else except possibly for the primes over $w$.  In
particular, $\rho_{f_i'}$ has type $C$ wherever $\rho_{f_{i-1}}$ does.  
Looking at the Hecke operators at primes above $p$ at which $\rho_{f_{i-1}}$
does not have type $C$ shows that $\rho_{f_i'}$ and $\rho_{f_{i-1}}$ have
the same type at these primes.

Now pick a finite, totally real, abelian extension $F_i/F_{i-1}$ satisfying
(1)-(6).  Let $f_i$ be the base change of $f_i'$ to $F_i$.  Then $F_i$ and
$f_i$ satisfy (1)-(9).  We have therefore established the lemma.
\end{proof}

We now return to the proof of the proposition.

\begin{proof}[Proof of Proposition~\ref{prop:lev}]
Replace $M$ with the compositum of $M$ and $\ker{\ol{\rho}_f}$.
To begin with, we may use Corollary~\ref{cor:sqr} to find a finite, totally
real, pre-solvable extension $F_1/F$ linearly disjoint from $M$ so that
$\det{\rho_f} \vert_{G_{F_1}}=\psi^2 \chi_p$.  Let $f_1=\psi^{-1} \cdot f
\vert_{F_1}$; this form has trivial nebentypus.  We can now apply
Lemma~\ref{lem:lower} to produce a finite, totally real, pre-solvable
extension $F_2/F_1$ linearly disjoint from $M$ and a form $f_2$ over $F_2$
so that:  $\ol{\rho}_{f_2} \cong \ol{\rho}_{f_1} \vert_{G_{F_2}}$;
$t_{f_1}(v) \approx t_{f_2}(v)$ at all places $v$ above $p$; $F_2$ has even
degree over $\Q$; the set of places of $F_2$ above $p$ at which $f_2$ has
type $C$ has even cardinality; $\ol{\rho}_{f_2}$ is everywhere
unramified; $\ol{\rho}_{f_2}$ is trivial at all places in $S$;
we have $\bN{v}=1 \pmod{p}$ for all $v$ above $S \setminus \Sigma_p$; and
$\pi_{f_2}$ is unramified except at the places above $p$ where $\rho_{f_2}$
has type $C$, where it is special of conductor one.  Choose a place $w$ of
$F_2$ above $S$ such that
\begin{displaymath}
\tr{\ol{\rho}}(\Frob_w) \ne \pm (1+\bN{w}),
\end{displaymath}
This is possible by \cite[Lemma~4.11]{DDT}, as before.  We now apply
Lemma~\ref{lem:raise2} to produce a finite, totally real, pre-solvable
extension $F_3/F_2$ linearly disjoint from $M$ and a form $f_3$ over $F_3$
such that:  $\ol{\rho}_{f_3}=\ol{\rho}_{f_2} \vert_{G_{F_3}}$;
$t_{f_3} \approx t \vert_{F_3}$ at all places $v$ above $p$; $w$ splits in
$F_3$.  Since $w$ satisfies the above condition and is split in $F_3$ it
follows that $f$ is not special at any place above $w$.  Therefore, we
can find a finite, totally real, pre-solvable extension $F_4/F_3$ linearly
disjoint from $M$ such that if $f_4=f_3 \vert_{F_4}$ then we can make
$\pi_{f_4}$ unramified at all places above $w$ and still satisfies the
other properties we want.  We now take $F'=F_4$ and $f'=\psi \cdot f_4$.
\end{proof}

\subsection
We now discuss some instances of Conjecture~\ref{conj:type} at places above
$p$.  These results will not be used in the remainder of the paper.  To
begin with, we have the following:

\begin{proposition}
Conjecture~\ref{conj:type} holds if $t_f$ only assumes the values $A$ and
$B$ above $p$ and $t$ only assumes the value $B$ above $p$.
\end{proposition}

\begin{proof}
See \cite[Theorem~3.5.7]{Kisin}.
\end{proof}

We now show, using some results of \S \ref{s:problems} and a
modularity lifting theorem of Skinner-Wiles, that in certain circumstances one
can switch between types $A$ and $C$.
We first need some terminology.  Let $F_v/\Q_p$ be a finite extension, let
 $\ol{\rho}_v:G_{F_v} \to \GL_2(\ol{\F}_p)$ be a representation and let
$\ast$ be $A$ or $C$.  We say that $(\ol{\rho}_v, \ast)$ is \emph{good} if
there exists a finite extension $F_v'/F_v$ and a lift $\rho_v$ of
$\ol{\rho}_v \vert_{G_{F_v'}}$ such that $\rho_v$ is weight two and has
definite type $\ast$ and we have
\begin{displaymath}
\ol{\rho}_v \vert_{G_{F_v'}}=\mat{\ol{\alpha}}{\ast}{}{\ol{\beta}}, \qquad
\rho_v=\mat{\alpha}{\ast}{}{\beta}
\end{displaymath}
with $\alpha$ reducing to $\ol{\alpha}$ and $\beta$ to $\ol{\beta}$ and
where $\ol{\alpha} \ne \ol{\beta}$.

\begin{proposition}
Conjecture~\ref{conj:type} holds if $\ol{\rho}_f$ satisfies (A1),
$t_f$ and $t$ only assume the values $A$ and $C$ at primes above $p$ and
for each place $v \mid p$ the pair $(\ol{\rho}_f \vert_{G_{F_v}}, t(v))$
is good.
\end{proposition}

\begin{proof}
For $v \mid p$ let $F_v'/F_v$ and $\rho_v$ be the extension and lift provided
by the goodness hypothesis.  Let $\tau_v$ be the inertial type of $\rho_v$
(see \S \ref{s:problems}).  Pick a finite, pre-solvable, totally real
 extension $F'/F$ such that $(F')_v=F'_v$ for $v \mid p$.  Define a lifting
problem (see \S \ref{s:problems}) $\ms{P}=(\Sigma, \psi,
t \vert_{F'}, \{\tau_v\})$ over $F'$ as follows.  The character $\psi$ is just
the nebentypus of $f$ restricted to $G_{F'}$.  The set $\Sigma$ consists of the
primes
at which $\ol{\rho}_f \vert_{G_{F'}}$ and $\psi$ ramify together with all the
primes over $p$.  The type function is specified.  Finally, the inertial
types $\tau_v$ have been specified for $v$ above $p$; for the other places,
just take $\tau_v$ to be the inertial type of $\rho_f \vert_{G_{F'}}$ at $v$.
By definition (and the fact that $t_f$ is definite), $\ms{P}$ is locally
solvable.  By Theorem~\ref{thm:lift} it
is therefore solvable.  Let $\rho'$ be a solution.  As $\ol{\rho}'=
\ol{\rho}_f \vert_{G_{F'}}$ the representation $\rho'$ is residually modular.
Furthermore, the conditions we have imposed ensures that it satisfies the
conditions of \cite[Theorem~5.1]{SkinnerWiles2}.  We therefore conclude that
$\rho'$ is modular, \ie, of the form $\rho_{f'}$.  The extension $F'/F$ and
the form $f'$ give the required output of Conjecture~\ref{conj:type}.
\end{proof}

\begin{remark}
The above proposition allows one to perform level raising or level lowering
at $p$ in many cases.  For instance, one can use it to find mod $p$ congruences
between forms which are special at $p$ and forms which are not.
\end{remark}

\section{Potential modularity}
\label{s:pmod}

\subsection
In \S \ref{s:pmod} we prove the following two potential modularity
theorems.  The second of these theorems is a stronger and more precise version
of Theorem~\ref{mainthm}.  These are simple modifications of Taylor's original
result \cite{Taylor3}.  In fact, our proof of potentially modularity is simpler
than Taylor's since we have a stronger modularity lifting theorem at out
disposal.

\begin{theorem}
\label{thm:pmod1}
Let $\ol{\rho}:G_F \to \GL_2(\ol{\F}_p)$ be any odd representation.  Let
$\psi:G_F \to \ms{O}^{\times}$ be a finite order character such that
$\det{\ol{\rho}}=\ol{\psi} \cdot \ol{\chi}_p$, let $M/F$ be a finite extension
and let $t$ be a definite type function on $\Sigma_p$.  Then there exists a
finite Galois extension $M'/F$ containing $M$ and a finite, totally real Galois
extension $F'/F$ linearly disjoint from $M'$ such that for any finite, totally
real extension $F''/F'$ linearly disjoint from $M'$ there exists a cuspidal
parallel weight two Hilbert eigenform $f$ over $F''$ with
coefficients in $\ol{\Q}_p$ such that $\det{\rho_f}=\psi \chi_p
\vert_{G_{F''}}$, $\ol{\rho}_f \cong \ol{\rho} \vert_{G_{F''}}$ and
$t_f \vert_{\Sigma_p}=t \vert_{F''}$.
\end{theorem}

\begin{theorem}
\label{thm:pmod2}
Let $\rho:G_F \to \GL_2(\ol{\Q}_p)$ be a finitely ramified, odd, weight
two representation such that $\ol{\rho}$ satisfies (A1) and (A2) of
\S \ref{s:reqt}.
Let $M/F$ be a finite extension.  Then there exists a finite Galois extension
$M'/F$ containing $M$ and a finite, totally real, Galois extension
$F'/F$ linearly
disjoint from $M'$ such that for any finite, totally real extension $F''/F'$
linearly disjoint from $M'$ there exists a cuspidal parallel weight two
Hilbert eigenform $f$ over $F''$ with coefficients in $\ol{\Q}_p$ such that
$\rho_f \cong \rho \vert_{G_{F''}}$.
\end{theorem}

In \S \ref{s:additional} we will improve these results and show that $F'$ can
be taken to be split at a given finite set of primes.

\subsection\label{ss:mb}
Following \cite{Taylor3}, we will prove Theorem~\ref{thm:pmod1} by using
a theorem of Moret-Bailly.  We now recall that theorem.  Let $F$ be any
number field.  A \emph{Skolem datum} over $F$ is a tuple $(X, \Sigma, \{L_v\}_
{v \in \Sigma}, \{\Omega_v\}_{v \in \Sigma})$ consisting of:
\begin{itemize}
\item A geometrically connected, separated, smooth scheme $X$ of finite
type over $F$.
\item A finite set of places $\Sigma$ of $F$.
\item For each $v \in \Sigma$ a Galois extension $L_v$ of $F_v$.
\item For each $v \in \Sigma$ a non-empty, Galois stable, open subset
$\Omega_v$ of $X(L_v)$.
\end{itemize}
Here we regard $X(L_v)$ as having the $v$-adic topology.  Our definition of
a Skolem datum is less general than the one given in \cite{MoretBailly}.

\begin{proposition}[\cite{MoretBailly}]
\label{prop:mb}
Let a Skolem datum as above be given.  Then there exists a finite Galois
extension $F'/F$ which splits over each $L_v$ (that is, $F' \otimes_F L_v$ is
a direct sum of $L_v$'s) and a point $x \in X(F')$ such that for each $v \in
\Sigma$ and each embedding $F' \to L_v$ the image of $x$ in $X(L_v)$ belongs
to $\Omega_v$.
\end{proposition}

It will be convenient to give a slightly improved version of the theorem here.
We learned of this improved version, and the proof, from
\cite[Proposition~2.1]{HSBT} (they state that L.~Dieulefait had made this
observation as well).

\begin{proposition}
\label{prop:mb2}
Let a Skolem datum be given and let $M/F$ be a finite extension.  Then one
can take the field $F'$ supplied by Proposition~\ref{prop:mb} to be linearly
disjoint from $M$.
\end{proposition}

\begin{proof}
As $X$ is smooth and geometrically connected, it has $F_v$-points for $v$
large (the Weil conjectures give mod $v$ points and then smoothness gives
$F_v$ points).  Choose a finite set of places $\Sigma'$ satisfying the
following:
\begin{itemize}
\item The set $\Sigma'$ is disjoint from $\Sigma$.
\item The extension $M/F$ is unramified everywhere above $\Sigma'$.
\item For $v \in \Sigma'$ the set $X(F_v)$ is non-empty.
\item The set $\{ \Frob_v \}_{v \in \Sigma'}$ generates $\Gal(M'/F)$, where
$M'$ is the Galois closure of $M$.
\end{itemize}
For $v \in \Sigma'$ let $L_v=F_v$ and $\Omega_v=X(F_v)$, which is non-empty
by construction.  Then $(X, \Sigma \cup \Sigma', \{L_v\}, \{\Omega_v\})$ is
a Skolem datum.  By Proposition~\ref{prop:mb} we get an extension $F'/F$ which
splits over each $L_v$ and a point $x \in X(F')$ such that for each embedding
$F' \to L_v$ the image of $x$ lands in $\Omega_v$.  As $F'$ splits over
$L_v=F_v$ for $v \in \Sigma'$ it follows that $F'$ is linearly disjoint from
$M$.
\end{proof}

\subsection
The following proposition is our main application of the theorem of
Moret-Bailly.

\begin{proposition}
\label{prop:potav}
Let $F$ and $K$ be totally real number fields.  Choose the following:
\begin{itemize}
\item A finite extension $M/F$.
\item Primes $\mf{p}$ and $\mf{l}$ of $K$ over distinct odd rational primes
$p$ and $\ell$.
\item Representations $\ol{\rho}_{\mf{p}}:G_F \to \GL_2(k_{\mf{p}})$
and $\ol{\rho}_{\mf{l}}:G_F \to \GL_2(k_{\mf{l}})$ with cyclotomic determinant.
\item Type functions $t_{\mf{p}}:\Sigma_p  \to \{A,B,C\}$ and $t_{\mf{l}}:
\Sigma_{\ell} \to \{A,B,C\}$.
\end{itemize}
We can then find a finite, totally real, Galois extension $F'/F$ linearly
disjoint from $M$ and a $\GL_2(K)$-type abelian variety $A/F'$ such that
$A[\mf{p}] \cong \ol{\rho}_{\mf{p}} \vert_{G_{F'}}$ and $T_{\mf{p}}A$ has
type $t_{\mf{p}} \vert_{F'}$, and similarly with $\mf{l}$ in place of
$\mf{p}$.
\end{proposition}

\begin{proof}
Pick a prime $\mf{r}$ of $K$ above a prime $r \ne p, \ell$ such that the
kernel of the reduction map $\GL_2(\ms{O}_K) \to \GL_2(k_{\mf{r}})$ is
torsion-free.  Pick a representation $\ol{\rho}_{\mf{r}}:G_F \to
\GL_2(k_{\mf{r}})$ with cyclotomic determinant, \eg, $\ol{\rho}_r=\ol{\chi}_r
\oplus 1$.  Let $V_w$ be the vector space schemes over $F$
corresponding to $\ol{\rho}_w$ for $w \in \{\mf{p}, \mf{l}, \mf{r}\}$.
For each $w$, fix an isomorphism $a_w:\bigwedge^2{V_w} \to k_w(1)$;
this is the same as fixing an isomorphism of $V_w$ with its Cartier dual.
Let $X$ be the scheme over $F$ classifying tuples $(A, i, \{\alpha_w\}_w)$
where:
\begin{itemize}
\item $A$ is a $\GL_2(K)$-type abelian variety, that is, an abelian variety
of dimension $2[K:\Q]$ together with an embedding of $\ms{O}_K$ into $\End(A)$.
\item $i:\ms{P}(A) \to \ms{D}^{-1}$ is an isomorphism of $\ms{O}_K$-modules
equipped with a notion of positivity.  Here $\ms{P}(A)$ is the set of
isogenies $A \to A^{\vee}$, its set of positive elements being the
polarizations; $\ms{D}^{-1}$ is the inverse different of $K$, its set of
positive elements being the totally real elements it contains.
\item $\alpha_w:A[w] \to V_w$ is an isomorphism of $\ms{O}$-modules under
which the $\ms{O}$-linear Weil pairing on the source coincides with the fixed
isomorphisms $a_w$, for each $w \in \{\mf{p}, \mf{l}, \mf{r}\}$.
\end{itemize}
Note that $X$ exists as a scheme (rather than as a stack) because the
kernel of $\GL_2(\ms{O}_K) \to \GL_2(\ms{O}_K/\mf{p} \mf{l} \mf{r})$ is
torsion-free.  For more details on the definition (which
will not be important to us), and the proof that $X$ exists, see
\cite[\S 1]{Rapoport}, \cite[\S 2]{Taylor3} or \cite[\S 4]{Taylor}.

We now sketch an argument to show that $X$ is smooth and geometrically
connected.  To begin with, there is a similar moduli problem
one can consider, which we will denote by $X'$, consisting of tuples $(A, i,
\{\alpha_w\}_w)$ where $A$ and $i$ are as above but $\alpha_w$
is an $\ms{O}$-linear isomorphism $A[w] \to k_w^2$.
We have not given all the details of the definition of the space $X'$, but it
is meant to be the same as the space $\ms{M}^{\ms{D}^{-1}}_n$ of
\cite[\S 1]{Rapoport}, with $n=\mf{p} \mf{l} \mf{r}$.
(Note that Rapoport only considers the case where $n$ is an integer, but the
arguments apply to the case of ideals of $\ms{O}_K$ as well).
The spaces $X$ and $X'$ are nearly twisted forms of each other, but not quite.
To elaborate on this, fix an embedding $F \to \C$ and trivializations of
each $(V_w)_{\C}$.  We then have a natural map
$X_{\C} \to X'_{\C}$.  This map is not surjective as in $X'$ the
only condition on $\alpha_w$ is that it be $\ms{O}$-linear while
in $X$ it must be $\ms{O}$-linear and also take the Weil pairing to the
given pairing $a_w$.  This is the only failure of surjectivity, however, and
one sees that $X_{\C} \to X'_{\C}$ identifies the former with the locus in the
latter where the
$\alpha_w$ take the Weil pairing to the $a_w$.  Using the computation of
the connected components of $X'_{\C}$ given in \cite[Theorem~1.28 (ii)]
{Rapoport}, ones thus sees that $X_{\C}$ is a component of $X'_{\C}$.  It is
therefore smooth since $X'_{\C}$ is (see the discussion of
\cite[\S 1]{Rapoport}, in particular \cite[Theorem~1.20]{Rapoport}).  We have
thus shown that $X/F$ is smooth an geometrically connected.

Let $\Sigma=\Sigma_p \cup \Sigma_{\ell} \cup \Sigma_{\infty}$.
We now define a finite Galois extension $L_v/F_v$ and a Galois stable open 
subset $\Omega_v$ of $X(L_v)$ for each $v \in \Sigma$, as follows:
\begin{itemize}
\item Say $v \in \Sigma_p$.  If $t_{\mf{p}}(v)$ equals $A$ (resp.\ $B$)
let $\ol{\Omega}_v$ be the subset of $X(\ol{F}_v)$ consisting of
those abelian varieties over $\ol{F}_v$ which have good reduction
and whose reduction is ordinary (resp.\ non-ordinary) at $\mf{p}$.  If
$t_{\mf{p}}(v)$ equals $C$ then take $\ol{\Omega}_v$ to be the subset of
$X(\ol{F}_v)$ consisting of those abelian varieties which have bad
reduction.  Let $L_v/F_v$ be any finite Galois extension for which
$\ol{\Omega}_v \cap X(L_v)$ is non-empty and take $\Omega_v$ to be this
intersection.
\item For $v \in \Sigma_{\ell}$ we define $L_v$ and $\Omega_v$ in exactly
the same manner.
\item For $v \in \Sigma_{\infty}$ we take $L_v=F_v$ and $\Omega_v=X(L_v)$.
One easily sees that $\Omega_v$ is non-empty.
\end{itemize}
We now have a Skolem datum $(X, \Sigma, \{L_v\}, \{\Omega_v\})$.
Proposition~\ref{prop:mb2} thus gives a finite Galois extension $F'/F$ linearly
disjoint
from $M$ and a point $x \in X(F')$ corresponding to a $\GL_2(K)$-type abelian
variety $A/F'$ satisfying the conclusions of the theorem.
\end{proof}

\subsection
We need one more result before proving Theorem~\ref{thm:pmod1}.

\begin{proposition}
\label{prop:univmod}
Let $F$ be a totally real field and $\ell$ an odd prime.  Then there
exists a finitely ramified, odd, weight two representation $\rho:G_F \to
\GL_2(\ol{\Q}_{\ell})$ and a finite Galois extension $M/F$ such that for any
finite totally real extension $F'/F$ linearly disjoint from $M$
the representation $\ol{\rho} \vert_{G_{F'}}$ satisfies (A1) and (A2) (with
$p$ changed to $\ell$) and $\rho \vert_{G_{F'}}$ comes from a cuspidal
parallel weight two Hilbert eigenform.
\end{proposition}

Before proving the proposition we require a lemma.

\begin{lemma}
Let $H/\Q$ be an imaginary quadratic extension
in which $\ell$ splits.  Let $\mf{l}^+$ and $\mf{l}^-$ be the two
primes above $\ell$.  We can then find a finitely ramified character
$\psi_0:G_H \to \ol{\Q}_{\ell}^{\times}$ such that $\psi_0 \vert_{I_{H,
\mf{l}^+}}=\chi_{\ell}$ and $\psi_0 \vert_{I_{H, \mf{l}^-}}=1$.
\end{lemma}

\begin{proof}
Let $\mf{r}$ be a prime of $H$ over a prime $r \ne \ell$ and put
$S=\{ \mf{l}^+, \mf{l}^-, \mf{r}\}$.  Let $U_{\mf{r}}$ be a compact open
subgroup of $\ms{O}_{H, \mf{r}}^{\times}$ such that $U_{\mf{r}} \cap
\ms{O}_H^{\times}=1$.  The map
\begin{displaymath}
\ms{O}_{H, \mf{l}^+}^{\times} \times \ms{O}_{H, \mf{l}^-}^{\times} \to
\A_H^{\times}/(H^{\times} H_{\infty}^{\times} U_{\mf{r}} \prod_{v \not \in S}
\ms{O}_{H, v}^{\times}).
\end{displaymath}
is injective and has open image.  It follows that any character of the
left group valued in $\ol{\Q}_{\ell}^{\times}$ extends to a character of
the right group, since $\ol{\Q}_{\ell}^{\times}$ is injective.
The result now follows from class field theory.
\end{proof}

We now return to the proof of the proposition.

\begin{proof}[Proof of Proposition~\ref{prop:univmod}]
Let $H/\Q$ be an imaginary quadratic field in
which $\ell$ splits completely.  Let $\mf{l}^+$ and $\mf{l}^-$ be the two
places of $H$ above $\ell$.  For each place $v$ of $F$ above $\ell$ label
the two places of $FH$ above $v$ corresponding to $\mf{l}^+$ and $\mf{l}^-$
as $v^+$ and $v^-$.  Let $\psi_0:G_H \to \ol{\Q}_{\ell}^{\times}$
be the character produced by the previous lemma.  We can multiply $\psi_0
\vert_{G_{FH}}$ by a finite order character of $G_{FH}$ to obtain a character
$\psi:G_{FH} \to \ol{\Q}_{\ell}^{\times}$ satisfying the following two
conditions:
\begin{itemize}
\item For each place $v$ of $F$ over $\ell$ we have $\psi \vert_{I_{FH, v^+}}
=\chi_{\ell}$ and $\psi \vert_{I_{FH, v^-}}=1$.
\item Let $\psi'$ be the $\Gal(FH/F)$ conjugate of $\psi$.  Then
$\ol{\psi} \vert_{G_{FH(\zeta_{\ell})}} \ne \ol{\psi}'
\vert_{G_{FH(\zeta_{\ell})}}$.
\end{itemize}
We put $\rho=\Ind_{FH}^F(\psi)$ and let $M$ be the Galois closure over
$F$ of the field determined by $\ker(\ol{\rho} \vert_{F(\zeta_{\ell})})$.

We now show that $\rho$ has the requisite properties.  It is
clear that $\rho$ is finitely ramified.  We can think of $\rho$ as
given by $\ol{\Q}_{\ell}[G_F] \otimes_{\ol{\Q}_{\ell}[G_{FH}]}
\ol{\Q}_{\ell}(\psi)$.  As such, if $g$ is any element of $G_F$ which does
not belong to $G_{FH}$ then $e_1=1 \otimes 1$ and $e_2=g \otimes 1$
form a basis for $\rho$.  In this basis, we have
\begin{equation}
\label{eq:mat1}
\rho(h)=\mat{\psi(h)}{}{}{\psi'(h)}, \qquad
\rho(g)=\mat{}{\psi(g^2)}{1}{}
\end{equation}
where $h \in G_{FH}$.  As any complex conjugation in $G_F$ can be taken to be
$g$, and these elements square to the identity, \eqref{eq:mat1}
shows that $\rho$ is odd.  We now check that $\rho$ has
weight two.  It suffices to check this after a finite extension, so we
may as well check that $\rho \vert_{G_{FH}}$ has weight two.  If $h$ is an
element of $I_{FH, v^+}$ then $g^{-1}hg$ belongs to $I_{H, v^-}$ and
conversely.  Thus \eqref{eq:mat1} shows that $\rho \vert_{I_{FH, v^+}}$ and
$\rho \vert_{I_{FH, v^-}}$ are isomorphic to $\chi_{\ell} \oplus 1$, and so
$\rho$ has weight two.

We now prove the statements about $\ol{\rho}$.  Choose an element $g$ of
$G_{F(\zeta_{\ell})} \subset G_F$ which
does not belong to $G_{FH}$, which is possible since $H$ is disjoint from
$F(\zeta_{\ell})$,
and let $e_1$, $e_2$ be a basis for $\rho$ as above.  The matrices
in \eqref{eq:mat1} belong to $\GL_2(\ol{\Z}_{\ell})$ and thus give an
integral model for $\rho$.  We define $\ol{\rho}$ to be the reduction of
this integral model.  From the first matrix in \eqref{eq:mat1} and our
second condition on $\psi$ we see that $\ol{\rho}(G_{FH(\zeta_{\ell})})$ is
contained in the group of diagonal matrices but not in the group of scalar
matrices.  On the other hand, we have $\ol{\rho}(g)=\mat{}{\ast}{1}{}$.  These
two conditions show that $\ol{\rho}(G_{F(\zeta_{\ell})})$ is an irreducible
subgroup of $\GL_2(\ol{\F}_{\ell})$.  Thus $\ol{\rho}$ satisfies (A1).  The
image of $\ol{\rho}$ is solvable and so (A2) is satisfied.

Now, since $\rho$ is the induction of a character it comes from an automorphic
form.  As $\rho$ is odd, weight two and absolutely irreducible, this
automorphic form is a cuspidal, parallel weight two Hilbert modular form.  Thus
$\rho$ is modular.

Finally, if $F'/F$ is any extension linearly disjoint from $M$ then
$\rho \vert_{G_{F'}}=\Ind_{F'H}^{F'}(\psi)$.  As $\psi \vert_{G_{HF'}}$
has the same two properties that we imposed on $\psi$, the above arguments
show that $\rho \vert_{G_{F'}}$ is modular and that $\ol{\rho} \vert_{G_{F'}}$
satisfies (A1) and (A2).  This completes the proof.
\end{proof}

\subsection
We can now complete proof of Theorem~\ref{thm:pmod1}.  Let $\ol{\rho}$, $\psi$,
$M_1=M$ and $t$ be given as in the statement of the theorem.  Pick an odd
prime $\ell \ne p$ and let $\rho'$ and $M_2$ be the representation and
extension given by Proposition~\ref{prop:univmod}.  Let $M'$ be the Galois
closure over $F$ of
the compositum of $M_1$ and $M_2$.  Write $\det{\rho'}=\chi_{\ell} \psi'$.
Pick a totally real field $K$ and primes $\mf{p}$ and $\mf{l}$ above $p$ and
$\ell$ such that the images of $\ol{\rho}$ and $\ol{\rho}'$ are contained in
$\GL_2(k_{\mf{p}})$ and $\GL_2(k_{\mf{l}})$.  Apply Corollary~\ref{cor:sqr} to
find a finite, totally real extension $F_1/F$, linearly disjoint from $M'$,
such that $\psi$ and $\psi'$ are squares when restricted to $G_{F_1}$; say
$\psi \vert_{G_{F_1}}=\psi_1^2$ and $\psi' \vert_{G_{F_1}}=(\psi'_1)^2$.
Put $\ol{\rho}_{\mf{p}}=\psi_1^{-1} \cdot (\ol{\rho} \vert_{G_{F_1}})$ and
$\ol{\rho}_{\mf{l}}=(\psi'_1)^{-1} \cdot (\ol{\rho}' \vert_{G_{F_1}})$.  Let
$t_{\mf{p}}=t$ and let $t_{\mf{l}}=t_{\rho'}$.

We now apply Proposition~\ref{prop:potav} (with the $M$ there being our
$M'$).  Let $F_2/F_1$ and $A/F_2$ be the resulting field and abelian variety.
Let $F''/F_2$ be any finite, totally real, extension linearly disjoint from
$M'$.  We have $A[\mf{l}] \vert_{G_{F''}}=(\psi_1')^{-1} \cdot \ol{\rho}'
\vert_{G_{F''}}$, which
satisfies (A1) and (A2).  Furthermore, $(T_{\mf{l}} A) \vert_{G_{F''}}$ and
$(\psi_1')^{-1} \cdot \rho' \vert_{G_{F''}}$ are each finitely ramified, odd,
weight two lifts of
$A[\mf{l}] \vert_{G_{F''}}$ which have the same type at each place of $F''$
over $\ell$ and $(\psi_1')^{-1} \cdot \rho' \vert_{G_{F''}}$ is modular.
We conclude from Theorem~\ref{thm:mlt} that $(T_{\mf{l}} A) \vert_{G_{F''}}$
is modular, that is, of the form $\rho_{f', \mf{l}'}$ for some parallel weight
two Hilbert eigenform $f'$ over $F''$ and some prime $\mf{l}'$ of its
coefficient field.  Since $(T_{\mf{l}} A) \vert_{G_{F''}}$
is absolutely irreducible, $f$ must be cuspidal.  By compatibility, we
see that $(T_{\mf{p}} A)_{G_{F''}}$ is isomorphic to $\rho_{f', \mf{p}'}$.
Letting $f=\psi_1 \cdot f'$, we find $\ol{\rho} \vert_{G_{F''}}=
\ol{\rho}_{f, \mf{p}'}$ and $\det{\rho_{f, \mf{p}'}}=\chi_p \psi
\vert_{G_{F''}}$.  Finally, since $\rho_{f, \mf{p}'}$ is off from
$(T_{\mf{p}} A) \vert_{G_{F''}}$ by a finite order character, the two have
the same type at all places over $p$.  Since the type of the Tate module
above $p$ is prescribed by $t$ by the way we chose $A$, so is the type of
$\rho_{f, \mf{p}'}$.  Taking $F'$ to be the Galois closure of $F_2$ over
$F$ finishes the proof.

\subsection
We now prove Theorem~\ref{thm:pmod2}.  Let $\rho$ and $M=M_1$ be given as in
the statement of the theorem.  Let $t$ be a definite type function on
$\Sigma_p$ with $t \approx t_{\rho} \vert_{\Sigma_p}$
Apply Theorem~\ref{thm:pmod1} to $\ol{\rho}$, $M$ and $t$.  Let $F'$
and $M_2$ be the resulting extensions.  Let $M'$ be the Galois closure
over $F$ of
the compositum of $M_1$, $M_2$ and the field determined by $\ker{\ol{\rho}}$.
Let $F''/F'$ be a finite, totally real extension
linearly disjoint from $M'$.  Let $f$ be a parallel weight two Hilbert
modular form over $F''$ with coefficients in $\ol{\Q}_p$ such that $\ol{\rho}
\vert_{G_{F''}}=\ol{\rho}_f$ and the type of $\rho_f$ at places above $p$
is given by $t$.  The representation $\ol{\rho} \vert_{G_{F''}}$ still
satisfies (A1) and (A2).  We can now apply
Theorem~\ref{thm:mlt} to conclude that $\rho \vert_{G_{F''}}$ is modular.
This completes the proof.

\subsection
We have now proved Theorem~\ref{mainthm}.  We will prove
Corollary~\ref{maincor} in \S \ref{s:conseq}.  We remark that
Proposition~\ref{prop:type} now applies and shows that $t_{\rho}$ is
definite for any odd, finitely ramified, weight two $\rho$ satisfying
(A1) and (A2).

\section{Finiteness of deformation rings}
\label{s:defring}

\subsection\label{ss:defring}
Let $F$ be a totally real field.  Fix an odd representation $\ol{\rho}:G_F \to
\GL_2(k)$ satisfying (A1) and (A2).  Also fix a finite order character
$\psi:G_F \to \ms{O}^{\times}$ such that $\det{\ol{\rho}}=\ol{\psi} \cdot
\ol{\chi}_p$.
Let $\Sigma$ be a finite set of primes of $F$, including all those above $p$
and those at which $\ol{\rho}$ or $\psi$ ramify.  For $v \in \Sigma$ let
$R^{\Box}_v$
denote the universal ring classifying framed deformations of $\ol{\rho}
\vert_{G_{F_v}}$ with determinant $\psi \chi_p$.  Choose a definite type
function $t$ defined on $\Sigma$ and for each $v \in \Sigma$ a quotient
$R^{\Box, \dag}_v$ of $R^{\Box}_v$.  We require the following:
\begin{enumerate}
\item $R^{\Box, \dag}_v$ is non-zero, reduced and $\ms{O}$-flat of dimension
$[F_v:\Q_p]+4$ (resp.\ 4) for $v \in \Sigma_p$ (resp.\ $v \not \in \Sigma_p$).
\item For each $v \in \Sigma$ there exists a finite Galois extension $L_v/F_v$
such that given an extension $E'/E$ and a point $R^{\Box, \dag}_v \to E'$
corresponding to a representation $\rho$ we have that $\rho \vert_{L_v}$ is
admissible of type $t(v)$.
\end{enumerate}
As usual, we will write $\Sigma_{\ast}$ for the inverse image of $\ast$
under $t$.

We now define the following deformation rings:
\begin{itemize}
\item We let $R^{\Box, \loc}$ be the tensor product of the $R^{\Box}_v$ for
$v \in \Sigma$.
\item We let $R^{\Box, \loc, \dag}$ be the tensor product of the $R^{\Box,
\dag}_v$ for $v \in \Sigma$.
\item We let $R^{\Box}_{\Sigma}$ be the universal ring of $\ol{\rho}$
classifying deformations of $\ol{\rho}$ with determinant $\psi \chi_p$ which
are unramified away from $\Sigma$ together with a framing at the places in
$\Sigma$.
\item We put $R^{\Box, \dag}=R^{\Box, \loc, \dag} \otimes_{R^{\Box, \loc}}
R^{\Box}_{\Sigma}$.
\item We let $R_{\Sigma}$ be the universal ring classifying deformations of
$\ol{\rho}$ with determinant $\psi \chi_p$ unramified away from $\Sigma$.
\item We let $R^{\dag}$ be the quotient of $R_{\Sigma}$ analogous to
$R^{\Box, \dag}$.
\end{itemize}
The main result of \S \ref{s:defring} is the following:

\begin{theorem}
\label{thm:defring}
The ring $R^{\dag}$ is finite over $\ms{O}$ of non-zero rank.  In particular,
the set $\Hom(R^{\dag}, \ol{\Q}_p)$ is finite and non-empty.
\end{theorem}

\subsection
We begin the proof of Theorem~\ref{thm:defring} with the following.

\begin{proposition}
\label{prop:def-ring-dim}
The ring $R^{\dag}$ has dimension at least one.
\end{proposition}

\begin{proof}
Since $\ol{\rho}$ satisfies condition (A1) we have $H^0(G_{F, S},
(\ad^{\circ}{\ol{\rho}})^*(1))=0$.  We can therefore apply
\cite[Proposition~4.1.5]{Kisin3} to conclude that there is an isomorphism
\begin{displaymath}
R^{\Box}_{\Sigma}=R^{\Box, \loc} \lbb x_1, \ldots, x_{r+n-1} \rbb/(f_1, \ldots,
f_{r+s})
\end{displaymath}
where $s=\sum_{v \mid \infty} \dim_k H^0(G_{F_v}, \ad^{\circ}{\ol{\rho}})$,
$n$ is the cardinality of $\Sigma$ and $r$ is some non-negative integer.  By
tensoring this isomorphism over $R^{\Box, \loc}$ with $R^{\Box, \loc, \dag}$,
we find
\begin{displaymath}
R^{\Box, \dag}=R^{\Box, \loc, \dag} \lbb x_1, \ldots, x_{r+n-1} \rbb/
(f_1, \ldots, f_{r+s})
\end{displaymath}
Now, the assumption that $\ol{\rho}$ is odd gives $s=[F:\Q]$.  The ring
$R^{\Box, \loc, \dag}$ has dimension $d+1$ where
\begin{displaymath}
d=\sum_{v \in \Sigma_p}([F_v:\Q_p]+3)+\sum_{v \in \Sigma \setminus \Sigma_p} 3
=[F:\Q]+3n
\end{displaymath}
We thus find
\begin{displaymath}
\dim{R^{\Box, \dag}} \ge ([F:\Q]+3n+1)+(r+n-1)-(r+s)=4n
\end{displaymath}
As $R^{\dag} \to R^{\Box, \dag}$ is formally smooth of relative dimension
$4n-1$ we conclude $\dim{R^{\dag}} \ge 1$.
\end{proof}

\subsection
In light of Proposition~\ref{prop:def-ring-dim}, to prove
Theorem~\ref{thm:defring} it suffices to show that $R^{\dag}$ is finite
over $\ms{O}$.  To do this we will need to use the following lemma:

\begin{lemma}
Let $F'/F$ be a finite extension for which $\ol{\rho} \vert_{G_{F'}}$ is
absolutely irreducible.  Let $\ol{R}_{\Sigma}$ be the universal ring
classifying deformations of $\ol{\rho} \vert_{G_{F'}}$ unramified outside of
the primes above $\Sigma$.  Then the map $\ol{R}_{\Sigma} \to R_{\Sigma}$ is
finite.
\end{lemma}

\begin{proof}
We sketch a proof.  As the rings involved are topologically finitely
generated, it suffices to show that the map is integral.  Deformation rings
are generated by the traces of elements of the group under the universal
representation, so the result thus follows from
the following statement:  if $\rho:G \to \GL_2(R)$ is a representation
and $g$ belongs to $G$ then $\tr{\rho(g)}$ satisfies a monic polynomial
with coefficients in the subring of $R$ generated by $\tr{\rho(g^{nk})}$,
for any fixed $n>1$ and varying $k$.  We leave this to the reader.
\end{proof}

\subsection
We now show the finiteness of $R^{\dag}$.  Apply Theorem~\ref{thm:pmod1}
to produce a finite, totally real extension $F'/F$ and a form $f$ over
$F'$ such that $\ol{\rho} \vert_{G_{F'}}$ satisfies (A1) and (A2),
$\ol{\rho}_f \cong  \ol{\rho} \vert_{G_{F'}}$ and
$t_f \vert_{\Sigma_{F', p}}=(t \vert_{F'}) \vert_{\Sigma_{F', p}}$.
By replacing $F'$ with a finite, totally real, pre-solvable extension,
and applying the same procedure as in the proof of Theorem~\ref{thm:mlt2}
(together with Proposition~\ref{prop:lev}) we may assume the above as well
as the the following:
\begin{itemize}
\item If $v'$ is a prime of $F'$ lying over a prime $v$ of $F$ belonging to
$\Sigma$ then $F'_{v'}$ contains $L_v$.
\item $\ol{\rho} \vert_{G_{F'}}$ is everywhere unramified.
\item $\ol{\rho} \vert_{G_{F'}}$ is trivial at every place above $\Sigma$.
\item $\rho_f$ is admissible of type $t(v)$ at all places $v$ above $\Sigma$.
\item $\psi \vert_{G_{F'}}=\psi_1^2$ and $\det{\rho_f}=\psi_2^2 \chi_p$
with $\psi_i$ unramified.
\item The set of primes of $F'$ lying over $\Sigma_C$ has even cardinality.
\item $F'$ has even degree over $\Q$.
\item $k$ contains the eigenvalues of the image of $\ol{\rho}$.
\end{itemize}
Replace $f$ by $\psi_2^{-1} \cdot f$ and let $\ol{\rho}_1$ be
$\psi_1^{-1} \cdot \ol{\rho} \vert_{G_{F'}}$.
We now define an deformation datum $\mc{D}^{\circ}=(t', \Sigma^{\ram},
\Sigma^{\aux})$ over $F'$ (with respect to $\ol{\rho}_1$) by taking
$\Sigma^{\ram}$ to be the set of primes
lying over $\Sigma \setminus \Sigma_p$, $t'$ to be the restriction of $t$ to
$F'$, $\Sigma^{\aux}$ to be $\{w\}$ where $w$ is any sufficiently large prime
satisfying the necessary conditions (see, for example, the proof of
Theorem~\ref{thm:mlt2}).  We use the notation of
\S \ref{ss:defrings} with an overline to indicate
the relevant deformation rings.  The above conditions fulfill the hypotheses
of Theorem~\ref{thm:ReqT}, so we conclude that $\ol{R}_{\mc{D}^{\circ}}$
is finite over $\ms{O}$.  Note that tensoring
by $\psi_1$ gives a bijection between deformations of $\ol{\rho}_1$ and
deformations of $\ol{\rho} \vert_{G_{F'}}$, and this bijection preserves any
quality which we care about.  In particular, it furnishes a natural map
$\ol{R}_{\Sigma(\mc{D}^{\circ})} \to R_{\Sigma}$ even though the first
ring classifies deformations with determinant $\chi_p$ while the second
classifies those with determinant $\psi \chi_p$.  The following lemma
completes the proof of Theorem~\ref{thm:defring}.

\begin{lemma}
The map $\ol{R}_{\Sigma(\mc{D}^{\circ})} \to R_{\Sigma} \to
R^{\dag}$ factors through $\ol{R}_{\mc{D}^{\circ}}$.
\end{lemma}

\begin{proof}
It suffices to prove the result on framed deformation rings.  For this,
it suffices to show that the map
\begin{displaymath}
\ol{R}^{\Box, \loc} \to R^{\Box, \loc} \to R^{\Box, \loc, \dag}
\end{displaymath}
factors through $\ol{R}^{\Box, \loc}_{\mc{D}^{\circ}}$.  Since
$R^{\Box, \loc, \dag}$ is reduced and flat over $\ms{O}$ it suffices to show
that if $E'/E$ is a finite extension and $R^{\Box, \loc, \dag} \to E'$ a point
then the resulting map $\ol{R}^{\Box, \loc} \to E'$ factors through
$\ol{R}^{\Box, \loc}_{\mc{D}^{\circ}}$.  This can easily be checked
place by place.
\end{proof}

\section{Lifting problems}
\label{s:problems}

\subsection
Let $F_v/\Q_{\ell}$ be a finite extension ($\ell=p$ allowed) and let $E$ be
any extension of $\Q_p$.  An \emph{inertial type} for $F_v$ is a representation
$\tau_v:I_{F_v} \to \GL_2(E)$ with open kernel.  We say that a representation
$\rho_v:G_{F_v} \to \GL_2(E)$ (assumed to be de Rham if $\ell=p$) has
\emph{inertial type} $\tau_v$ if the Weil-Deligne representation associated to
$\rho_v$ is isomorphic to $\tau_v$ when restricted to inertia.  The ``inertial
type'' knows nothing about the monodromy operator of the Weil-Deligne
representation.  For example, in the $\ell \ne p$ case ``inertial type'' does
not distinguish between unramified representations and representations with
unipotent inertia.  Do not confuse ``inertial type'' and ``type'':  they
are very different.  In fact, by Proposition~\ref{prop:type}, they are
complementary (at least in global situations):  ``type'' determines the data
of the monodromy operator, exactly what ``inertial type'' forgets.

\subsection
We now apply Theorem~\ref{thm:defring} to study certain lifting problems.
Let $\ol{\rho}:G_F \to \GL_2(\ol{\F}_p)$ be a given representation.  A
\emph{lifting problem} is a tuple $\ms{P}=(\Sigma, \psi, t, \{\tau_v\}_{v
\in \Sigma})$ consisting of:
\begin{itemize}
\item A finite set $\Sigma$ of finite places of $F$, including all those at
which $\ol{\rho}$ ramifies and all those over $p$.
\item A finite order character $\psi:G_F \to \ol{\Q}_p^{\times}$, unramified
outside $\Sigma$, such that $\det{\ol{\rho}}=\ol{\psi} \cdot \ol{\chi}_p$.
\item A definite type function $t$ defined on $\Sigma$.
\item For each $v \in \Sigma$ an inertial type $\tau_v$.
\end{itemize}
A \emph{solution} to a lifting problem $\ms{P}$ is a representation $\rho:G_F
\to \GL_2(\ol{\Q}_p)$ satisfying the following:
\begin{itemize}
\item $\rho$ is a weight two lift of $\ol{\rho}$.
\item $\det{\rho}=\psi \chi_p$.
\item $\rho$ is unramified outside of $\Sigma$.
\item $\rho \vert_{G_{F_v}}$ has type $t(v)$ for $v \in \Sigma$.
\item $\rho \vert_{G_{F_v}}$ has inertial type $\tau_v$ for $v \in \Sigma$.
\end{itemize}
A \emph{local solution} to a lifting problem $\ms{P}$ is a family
$\{\rho_v\}_{v \in \Sigma}$ consisting of representations $\rho_v:G_{F_v} \to
\GL_2(\ol{\Q}_p)$ which satisfy the following:
\begin{itemize}
\item The reduction of some stable lattice in $\rho_v$ is isomorphic to
$\ol{\rho} \vert_{G_{F_v}}$.
\item $\det{\rho_v}=\psi \chi_p \vert_{G_{F_v}}$.
\item For $v$ over $p$ the representation $\rho_v$ has weight two.
\item $\rho_v$ has definite type $t(v)$.
\item $\rho_v$ has inertial type $\tau_v$.
\end{itemize}
Note that we require $\rho_v$ to have \emph{definite type} in the above
that is, it cannot have type $A/C$ or $AB/C$.  The main result of
\S \ref{s:problems} is the following theorem.

\begin{theorem}
\label{thm:lift}
Let $\ol{\rho}:G_F \to \GL_2(\ol{\F}_p)$ be an odd representation satisfying
(A1) and (A2) and let $\ms{P}$ be a lifting problem.  Then there are only
finitely many solutions to $\ms{P}$.  A solution exists if and only if
a local solution exists.
\end{theorem}

To prove this theorem, we will use Theorem~\ref{thm:defring}.  The work lies
in showing
that the necessary local deformation rings exist.  This has basically
been done already by Kisin and Gee, though we will need some variants on their
work.  We establish the necessary results in the next two sections.  After
proving Theorem~\ref{thm:lift} we apply it to give a proof of
Theorem~\ref{mainthm2}.

\subsection
We now establish some results about local deformation rings in the unequal
characteristic case.  Let
$F_v/\Q_{\ell}$ be a finite extension with $\ell \ne p$, let $\ol{\rho}_v:
G_{F_v} \to \GL_2(k)$ be a representation and let $\psi_v:G_{F_v} \to
\GL_2(\ms{O})$ be a finite order character with $\det{\ol{\rho}_v}=\ol{\psi}_v
\cdot \ol{\chi}_p$.  Let $R^{\Box, \psi_v}_v$ denote the universal ring
classifying framed deformations of $\ol{\rho}_v$ with determinant $\psi_v
\chi_p$.  The main result we need is the following:

\begin{proposition}
\label{prop:lring2}
Let $\tau_v$ be an inertial type for $F_v$ and let $\ast$ be $AB$ or $C$.
There exists a quotient $R_v^{\Box, \psi_v, \tau_v, \ast}$ of $R_v^{\Box,
\psi_v}$ with the following properties:
\begin{enumerate}
\item $R_v^{\Box, \psi_v, \tau_v, \ast}$ is $\ms{O}$-flat, reduced and all of
its components have dimension 4.
\item Let $x:R_v^{\Box, \psi_v} \to E'$ be a map corresponding to a
representation $\rho$.  If $x$ factors through $R^{\Box, \psi_v, \tau_v,
\ast}_v$ then $\rho$ has type $\ast$ and inertial type $\tau_v$.  Conversely,
if $\rho$ has definite type $\ast$ and inertial type $\tau_v$ then $x$ factors
through $R^{\Box, \tau_v, \ast}_v$.
\end{enumerate}
The ring $R^{\Box, \psi_v, \tau_v, \ast}_v$ may be zero.
\end{proposition}

The subtlety in (2) above is that if $\rho$ has \emph{indefinite
type} (\ie, type $AB/C$) then we cannot conclude that $x$ factors through
$R^{\Box, \psi_v, \tau_v, \ast}_v$.  We need a few lemmas before proving the
proposition.

\begin{lemma}
\label{lem:lring2a}
Let $\tau_v$ be an inertial type for $F_v$.  There exists a quotient
$R^{\Box, \psi_v, \tau_v}_v$ of $R^{\Box, \psi_v}_v$ with the following
properties:
\begin{enumerate}
\item $R^{\Box, \psi_v, \tau_v}_v$ is $\ms{O}$-flat, reduced and all of its
components have dimension 4.
\item Let $x:R^{\Box, \psi_v}_v \to E'$ be a map corresponding to a
representation $\rho$.  Then $x$ factors through $R^{\Box, \psi_v, \tau_v}_v$
if and only if $\rho$ has inertial type $\tau_v$.
\end{enumerate}
The ring $R^{\Box, \psi_v, \tau_v}_v$ may be zero.
\end{lemma}

\begin{proof}
This is exactly \cite[Proposition~3.1.3]{Gee2}.
\end{proof}

We now compute the components of the ring $R_v^{\Box, \psi_v, \tau_v}$ in
certain cases.

\begin{lemma}
\label{lem:lring2b}
Assume that $\ol{\rho}$ is trivial, that $F_v$ contains the $p$th roots of
unity, let $\tau_v$ be the trivial inertial type and let $\psi_v$ be the
trivial character.  The ring $R^{\Box, \psi_v, \tau_v}_v$ then has two
components, corresponding to admissible deformations of types $AB$ and $C$.
\end{lemma}

\begin{proof}
Let $\rho$ be a deformation of $\ol{\rho}$ to an extension $E'/E$.
One easily sees
\begin{displaymath}
\textrm{$\rho$ is admissible} \iff \textrm{$\rho$ is semi-stable with
determinant $\chi_p$.}
\end{displaymath}
Here ``admissible'' means admissible of type $AB$ or $C$ and ``semi-stable''
means inertia acts unipotently.  The direction $\implies$ shows that if $x$
is an $E'$ point of $R_v^{\Box, \psi_v}$ which factors through one of the rings
$R_v^{\Box, \ast}$ of Proposition~\ref{prop:lring} then $x$ factors through
$R_v^{\Box, \psi_v, \tau_v}$.  We thus have a closed immersion
\begin{displaymath}
\Spec(R_v^{\Box, \ast}) \to \Spec(R_v^{\Box, \psi_v, \tau_v}).
\end{displaymath}
Proposition~\ref{prop:lring} shows that $R_v^{\Box, \ast}$ is integral and
four dimensional and so $R_v^{\Box, \ast}$ must be a component of
$R_v^{\Box, \psi_v, \tau_v}$.  To show that $R_v^{\Box, AB}$ and $R_v^{\Box,
C}$ are all the components of $R_v^{\Box, \psi_v, \tau_v}$ it suffices to show
that any $E'$-point of the latter lies on one of the former.  This follows
from the implication $\impliedby$.
\end{proof}

We now prove the proposition.

\begin{proof}[Proof of Proposition~\ref{prop:lring2}]
Let $F_v'$ be the compositum of the extensions determined by $\ker{\tau_v}$,
$\ker{\ol{\rho}_v}$, $\ker{\psi_v}$ and the extension $F_v(\zeta_p)$.  Let
$\ol{R}^{\Box, \psi_v}_v$ be the universal ring classifying framed
deformations of $\ol{\rho} \vert_{G_{F'}}$ with determinant $\psi_v \chi_p
=\chi_p$ and let $\ol{R}^{\Box, \psi_v, \tau_v}_v$ be quotient
considered in the previous lemma.  We have a natural map
$\Spec(R^{\Box, \psi_v, \tau_v}_v) \to \Spec(\ol{R}^{\Box, \psi_v, \tau_v}_v)$.
By the previous lemma, we can write $\Spec(\ol{R}^{\Box, \psi_v, \tau_v}_v)
=X_{AB} \cup X_C$ where $X_{\ast}$ is a components and an $E'$-point
of $\Spec(\ol{R}^{\Box, \psi_v, \tau_v}_v)$ lies on $X_{\ast}$ if and only if
it has type $\ast$.  We can then take $\Spec(R^{\Box, \psi_v, \tau_v, \ast}_v)$
to be the union of the components of $\Spec(R^{\Box, \psi_v, \tau_v}_v)$ which
map into $X_{\ast}$.  As $R^{\Box, \psi_v, \tau_v, \ast}_v$ is a union of
components, it satisfies (1) in the statement of the proposition.  The
verification of (2) is easy and left to the reader.
\end{proof}

\subsection
We now establish some results about local deformations ring in the equal
characteristic case.  Let $F_v/\Q_p$ be a finite extension, let $\ol{\rho}_v:
G_{F_v} \to \GL_2(k)$ be a representation and let $\psi_v:G_{F_v} \to
\ms{O}^{\times}$ be a finite order character with $\det{\ol{\rho}_v}=
\ol{\psi}_v \cdot \ol{\chi}_p$.  Let $R^{\Box, \psi_v}_v$ denote the universal
ring classifying framed deformations of $\ol{\rho}_v$ with determinant $\psi_v
\chi_p$.  The main result we need is the following:

\begin{proposition}
\label{prop:pring2}
Let $\tau_v$ be an inertial type for $F_v$ and let $\ast$ be $A$, $B$ or $C$.
There exists a quotient $R^{\Box, \psi_v, \tau_v, \ast}_v$ of $R^{\Box,
\psi_v}_v$ with the following properties:
\begin{enumerate}
\item $R^{\Box, \psi_v, \tau_v, \ast}_v$ is $\ms{O}$-flat, reduced and all of
its components have dimension $[F_v:\Q_p]+4$.
\item Let $x:R_v^{\Box, \psi_v} \to E'$ be a map corresponding to a
representation $\rho$.  If $x$ factors through $R^{\Box, \psi_v, \tau_v,
\ast}_v$ then $\rho$ has type $\ast$ and inertial type $\tau_v$.  Conversely,
if $\rho$ has definite type $\ast$ and inertial type $\tau_v$ then $x$ factors
through $R^{\Box, \psi_v, \tau_v, \ast}_v$.
\end{enumerate}
The ring $R^{\Box, \psi_v, \tau_v, \ast}_v$ may be zero.
\end{proposition}

The proof follows the same basic plan as that of Proposition~\ref{prop:lring2}
and so we omit some of the details.  We begin with the analogue of
Lemma~\ref{lem:lring2a}.

\begin{lemma}
Let $\tau_v$ be an inertial type for $F_v$.  There exists a quotient
$R^{\Box, \psi_v, \tau_v}_v$ of $R^{\Box, \psi_v}_v$ with the following
properties:
\begin{enumerate}
\item $R^{\Box, \psi_v, \tau_v}_v$ is $\ms{O}$-flat, reduced and all of its
components have dimension $[F_v:\Q_p]+4$.
\item Let $x:R^{\Box, \psi_v}_v \to E'$ be a map corresponding to a
representation $\rho$.  Then $x$ factors through $R^{\Box, \psi_v, \tau_v}_v$
if and only if $\rho$ is weight two and has inertial type $\tau_v$.
\end{enumerate}
The ring $R^{\Box, \psi_v, \tau_v}_v$ may be zero.
\end{lemma}

\begin{proof}
This follows from \cite[Theorem~3.3.4]{Kisin4}.
\end{proof}

We now compute the components of $R^{\Box, \psi_v, \tau_v}_v$ in certain
cases, as in Lemma~\ref{lem:lring2b}.

\begin{lemma}
Assume that $\ol{\rho}$ is trivial, that $F_v$ contains the $p$th roots of
unity, let $\tau_v$ be the trivial inertial type and let $\psi_v$ be the
trivial character.  The ring $R^{\Box, \psi_v, \tau_v}_v$ then has three
components, corresponding to admissible deformations of types $A$, $B$ and $C$.
\end{lemma}

\begin{proof}
The proof is essentially the same as that of Lemma~\ref{lem:lring2b}.
Proposition~\ref{prop:pring} shows that the space of deformations of
$\ol{\rho}$ which are admissible of a given type form a component.
An $E'$-point of $R^{\Box, \psi_v, \tau_v}_v$ gives rise to a semi-stable
representation with determinant $\chi_p$ reducing to the trivial
representation, and any such representation is admissible.  Thus the stated
components are all of the components.
\end{proof}

The proof of Proposition~\ref{prop:pring2} now follows exactly as the
proof of Proposition~\ref{prop:lring2}.

\subsection
We now prove the Theorem~\ref{thm:lift}.  A global solution of $\ms{P}$
gives rise to a local solution, since we know that a global solution has
definite type (Theorem~\ref{mainthm} and Proposition~\ref{prop:type}).
Thus if no local solution to $\ms{P}$ exists then no solution exists.  Assume
now that $\ms{P}$ admits a local solution.
Let $E/\Q_p$ be a large finite extension and let $\ms{O}$ be its ring
of integers.  For $v \in \Sigma$ we define $R_v^{\dag}$ to be the ring
$R_v^{\Box, \psi_v, \tau_v, t(v)}$ constructed in the previous sections, where
$\psi_v=\psi \vert_{G_{F_v}}$.  These
rings are non-zero due to the existence of a local solution.  They satisfy
the conditions specified in \S \ref{ss:defring} by
Proposition~\ref{prop:lring2} and Proposition~\ref{prop:pring2}.
Let $R^{\dag}$ be the ring defined in \S \ref{ss:defring} associated to the
above local deformation rings.  We then see that any $\ol{\Q}_p$-point of
$R^{\dag}$ gives a solution to $\ms{P}$, and all solutions are of this form.
Theorem~\ref{thm:defring} gives the desired result.

\subsection
We now give a simplified version of Theorem~\ref{thm:lift}, which is easier to
apply in many circumstances.  Before doing this, we introduce some terminology.
Let $F_v/\Q_{\ell}$ be a finite extension ($\ell=p$ allowed) and let
$\ol{\rho}_v:G_{F_v} \to \GL_2(\ol{\F}_p)$ be a representation.  We say that a
definite type $\ast$ is \emph{compatible} with $\ol{\rho}_v$ if $\ol{\rho}_v$
admits a lift to $\ol{\Q}_p$ (required to be weight two if $v \mid p$) which is
definite of type $\ast$.  Similarly,
for a representation $\ol{\rho}:G_F \to \GL_2(\ol{\F}_p)$ with $F/\Q$ finite,
we may speak of the compatibility of $\ol{\rho}$ with a definite type function.
We now have the following result:

\begin{theorem}
\label{thm:lift2}
Let $\ol{\rho}:G_F \to \GL_2(\ol{\F}_p)$ be an odd representation
satisfying (A1) and (A2), let $\psi:G_F \to \ol{\Q}_p^{\times}$ be a finite
order character such that $\det{\ol{\rho}}=\ol{\psi} \cdot \ol{\chi}_p$, let
$\Sigma$ be a finite set of places including all those at which $\ol{\rho}$ or
$\psi$ ramify and those above $p$, let $\Sigma'$ be a subset of $\Sigma$ and
let $t$ be a definite type function on $\Sigma'$ compatible with $\ol{\rho}$.
Then $\ol{\rho}$ admits a weight two lift to $\ol{\Q}_p$ which is unramified
outside of $\Sigma$, has determinant $\psi \chi_p$ and has type $t$ at the
places in $\Sigma'$.
\end{theorem}

To prove Theorem~\ref{thm:lift2} we need to know that every residual
representation is compatible with some type.  We will prove this (and in
fact more precise results) in the following two sections.  We will then
return to the proof of Theorem~\ref{thm:lift2}.

\subsection
Let $F_v/\Q_{\ell}$ be a finite extension with $\ell \ne p$.
The purpose of this section is to prove the following proposition:

\begin{proposition}
\label{prop:llift}
Given a representation $\ol{\rho}_v:G_{F_v} \to \GL_2(\ol{\F}_p)$ there
exists a representation $\rho_v:G_{F_v} \to \GL_2(\ol{\Q}_p)$ which lifts
$\ol{\rho}_v$, has the same conductor as $\ol{\rho}_v$ and has definite
type.  In particular, $\ol{\rho}_v$ is compatible with some type.
\end{proposition}

\begin{proof}
Write $\ol{\rho}$ in place of $\ol{\rho}_v$.
Let $G=G_{F_v}$, let $I=I_{F_v}$ be the inertia subgroup and let $I^w$ be the
wild inertia subgroup.  Let $I'$ be the closed normal subgroup of $G$
containing $I^w$ and whose image in $G/I^w$ is the prime-to-$p$ part of
tame inertia.
The representation $\ol{\rho} \vert_{I'}$ is semi-simple since $I'$ is
prime-to-$p$.  We consider three cases:
(1) $\ol{\rho} \vert_{I'}$ is irreducible; (2) $\ol{\rho} \vert_{I'}=
\ol{\alpha} \oplus \ol{\beta}$ with $\ol{\alpha} \ne \ol{\beta}$; and (3)
$\ol{\rho} \vert_{I'}= \ol{\alpha} \oplus \ol{\alpha}$.  We first show that
a definite type lift exists in each of these cases and then refine this
result to produce a definite type lift with equal conductor.

{\it Case 1.} As $I'$ has no mod $p$ cohomology, the Hochschild-Serre spectral
sequence gives $H^2(G, \ad^{\circ}{\ol{\rho}})=H^2(G/I',
(\ad^{\circ}{\ol{\rho}})^{I'})$.  The latter group vanishes since
$(\ad^{\circ}{\ol{\rho}})^{I'}=0$.  Thus there is no obstruction to
deforming $\ol{\rho}$ and so we can lift it to  $\ol{\Q}_p$.  Furthermore, we
find a lift whose determinant is any prescribed lift of the determinant of
$\ol{\rho}$.  We can thus find a lift $\rho$ whose determinant is a finite order
character times the cyclotomic character.  Since $\rho$ is also irreducible
it follows that $\rho$ is definite of type $AB$.

{\it Case 2.} The group $G$ must permute the characters $\ol{\alpha}$ and
$\ol{\beta}$.
First assume that $G$ permutes them trivially.  Then both extend to characters
of $G$ and $\ol{\rho}=\ol{\alpha} \oplus \ol{\beta}$.  Pick characters
$\alpha$, $\beta$, $\gamma$, $\delta$ of $G$ taking values in
$\ol{\Q}_p^{\times}$ as follows:  $\alpha$ and $\beta$ are finite order and
lift $\ol{\alpha}$ and $\ol{\beta}$; $\gamma$ is unramified of infinite order
and with trivial reduction; and $\delta$ is a finite order character reducing
to the cyclotomic character.  Then
$\rho=(\alpha \gamma \delta^{-1} \chi_p) \oplus (\beta \gamma^{-1})$ is a lift
of $\ol{\rho}$ which is definite of type $AB$.

Now assume that $G$ permutes $\alpha$ and $\beta$ non-trivially.  Let $H$ be
the subgroup of $G$ which stabilizes $\ol{\alpha}$ and $\ol{\beta}$.  Then
$\ol{\alpha}$ and $\ol{\beta}$ extend to characters of $H$ and $\ol{\rho}$ is
the induction of either of them from $H$ to $G$.  Let $\alpha:H \to
\ol{\Q}_p^{\times}$ be a finite order lift of $\ol{\alpha}$ and let
$\gamma:G \to \ol{\Q}_p^{\times}$ be a character which reduces to the trivial
character and is such that $\gamma^2$ is a finite order character times the
cyclotomic character.  Then $\rho=\Ind_H^G(\alpha \gamma)$ is a definite type
$AB$ representation lifting $\ol{\rho}$.  (It does not have type $C$ since it
is irreducible.)

{\it Case 3.} The character $\ol{\alpha}$ is fixed by $G$ and thus extends to
$G$.  We can thus twist $\ol{\rho}$ by $\ol{\alpha}^{-1}$ and assume that
$\ol{\rho}(I')=1$.  We can therefore regard $\ol{\rho}$ as a representation
of $G/I'$.  This group is topologically generated by $F$ and $N$ subject to
the relation $FNF^{-1}=N^q$, where $q$ is the cardinality of the residue field
of $F_v$, and the topological condition that $N$ is pro-$p$.  If
$\ol{\rho}(N)=1$ then $\ol{\rho}$ is unramified and it is easy to produce
a definite lift.  Thus assume $N \ne 0$.  We can then pick a basis so that
\begin{displaymath}
\ol{\rho}(N)=\mat{1}{\ast}{}{1}
\end{displaymath}
It follows from the relation $FNF^{-1}=N^q$ that
\begin{displaymath}
\ol{\rho}(F)=\mat{qa}{\ast}{}{a}.
\end{displaymath}
We can twist $\ol{\rho}$ by an unramified character so that $a=1$.  We then
find that $\ol{\rho}$ is a non-trivial extension of 1 by $\ol{\chi}_p$.  Say
that $\ol{\rho}$ is defined over the finite field $k$ and let $\ms{O}$ be the
Witt vectors of $k$.  Consider the map
\begin{displaymath}
H^1(G_{F_v}, \ms{O}(\chi_p)) \to H^1(G_{F_v}, k(\chi_p)).
\end{displaymath}
It is surjective by Kummer theory; in fact, every element in the
target is the image of a non-torsion element of the source.  The representation
$\ol{\rho}$ is represented by a class $\ol{c}$ in the target.  Let $c$ be
a non-torsion class in the source lifting $\ol{c}$.  Then the image of
$c$ in $H^1(G_{F_v}, E(\chi_p))$ is non-zero and corresponds to a definite
type $C$ lift of $\ol{\rho}$.

{\it Controlling the conductor.} We now show how one can refine the above
results to produce definite type lifts of $\ol{\rho}$ with the same conductor
as $\ol{\rho}$.  Observe that if $\rho$ is a lift of $\ol{\rho}$ then
\begin{displaymath}
f(\ol{\rho})-f(\rho)=\dim(\ol{\rho}^I) - \dim(\rho^I) \ge 0
\end{displaymath}
where $f$ is the exponent of the conductor.  Now, if $\dim(\ol{\rho}^I)=2$
then $\ol{\rho}$ is unramified.  It is not hard to produce a definite
type $AB$ lift $\rho$ which is unramified, which shows that one can find
a definite type lift of $\ol{\rho}$ with the same conductor.

Now consider the case where $\dim(\ol{\rho}^I)=1$.  We then have that
$\ol{\rho}^{I'}$ is either one or two dimensional.  First consider the
case where $\dim(\ol{\rho}^{I'})=1$.  Then $\ol{\rho} \vert_{I'}=\ol{\alpha}
\oplus 1$ where $\ol{\alpha}$ is non-trivial.  As in Case 2 we find that
$\ol{\alpha}$ extends to $G$ and $\ol{\rho}=\ol{\alpha} \oplus 1$.  The
definite type $AB$ lift $\rho$ produced in Case 2 (using $\beta=1$) has
$\dim(\ol{\rho}^I)=1$
and thus has the same conductor as $\ol{\rho}$.  Now consider the case where
$\dim(\ol{\rho}^{I'})=2$ so that $\ol{\rho}(I')=1$.  We then find that
$\ol{\rho} \vert_I$ is a non-trivial extension of 1 by $\ol{\chi}_p$.  The
definite type $C$ lift $\rho$ produced in Case 3 above is such that $\rho
\vert_I$ is a non-trivial extension of 1 by $\chi_p$.  We thus find
$\dim(\rho^I)=1$ which shows that $\rho$ and $\ol{\rho}$ have the same
conductor.  This completes the proof.
\end{proof}

\subsection
Now let $F_v/\Q_p$ be a finite extension.
The purpose of this section is to prove the following proposition:

\begin{proposition}
\label{prop:plift}
Given a representation $\ol{\rho}_v:G_{F_v} \to \GL_2(\ol{\F}_p)$ there
exists a representation $\rho_v:G_{F_v} \to \GL_2(\ol{\Q}_p)$ which lifts
$\ol{\rho}_v$, has weight two and has definite type.  In particular,
$\ol{\rho}_v$ is compatible with some type.
\end{proposition}

We first need a lemma.

\begin{lemma}
\label{lem:plift}
Let $F_v'$ be a quadratic extension of $F_v$.  The
there exists a character $\psi:G_{F_v'} \to \ol{\Q}_p^{\times}$ with
the following two properties:  (1) $\psi$ and $\psi'$ are de Rham with
non-positive Hodge-Tate weights; and (2) $\psi \cdot \psi'$ is a finite order
character times the cyclotomic character.  Here $\psi'$ is the conjugate of
$\psi$ by $\Gal(F_v'/F_v)$.
\end{lemma}

\begin{proof}
Let $S$ be the set of all embeddings of $F_v'$ into $\ol{\Q}_p$.  Each
element of $i \in S$ yields, via class field theory, a character
$\gamma_i:W_{F_v'} \to \ol{\Q}_p^{\times}$ where $W_{F_v'}$ is the Weil
group of $F_v'$.  These characters are de Rham with
non-positive Hodge-Tate weights.  Pick a subset $S_0$ of $S$ such that $S$ is
the disjoint union of $S_0$ and $\sigma S_0$, where $\sigma$ is the non-trivial
element of $\Gal(F_v'/F_v)$.  Let $\phi$ be the product of the $\gamma_i$
with $i \in S_0$.  The conjugate $\phi'$ is then the product of the $\gamma_i$
with $i \in \sigma S_0$.  The characters $\phi$ and $\phi'$ are de Rham with
non-positive Hodge-Tate weights since each $\gamma_i$ is.  Furthermore, $\phi
\cdot \phi'$ corresponds via class field theory to
the composite $(F_v')^{\times} \to \Q_p^{\times} \to \ol{\Q}_p^{\times}$,
where the first map is the norm map and the second is the canonical inclusion.
We thus see that $\phi \cdot \phi'$ agrees with the cyclotomic character on
inertia.  It follows that we can take $\psi$ to be $\phi$ times an
appropriate unramified character.
\end{proof}

We now return to the proof of the proposition.

\begin{proof}[Proof of Proposition~\ref{prop:plift}]
Write $G$ in place of $G_{F_v}$.  We first consider the case where $\ol{\rho}
=\ol{\rho}_v$ is reducible.  Say that $\ol{\rho}$ is defined over the finite
field $k$ and let $\ms{O}$ be the Witt vectors of $k$.  By twisting we can
assume $\ol{\rho}$ has the form
\begin{displaymath}
\mat{\ol{\alpha} \cdot \ol{\chi}_p}{\ast}{}{1}.
\end{displaymath}
Thus $\ol{\rho}$ corresponds to a class $\ol{c}$ of $H^1(G, k(\ol{\alpha} \cdot
\ol{\chi}_p))$.  If $\ol{\alpha}=1$ then, as was the case in the proof of 
Proposition~\ref{prop:llift} we can find a non-torsion class $c$ in
$H^1(G, \ms{O}(\chi_p))$ lifting the class $\ol{c}$.  The representation
corresponding to $c$ is then a definite type $C$ lift of $\ol{\rho}$.  Now
say $\ol{\alpha} \ne 1$.  Let $\alpha$ be a finite order lift of $\ol{\alpha}$
and let $\gamma:G \to \ms{O}^{\times}$ be an infinite order unramified
character reducing to the trivial character.
We have an exact sequence
\begin{displaymath}
0 \to (\ms{O}/\mf{m})(\gamma^2 \alpha) \to (\ms{O}/\mf{m}^{n+1})(\gamma^2
\alpha) \to (\ms{O}/\mf{m}^n)(\gamma^2 \alpha) \to 0
\end{displaymath}
which yields
\begin{displaymath}
H^1(G, (\ms{O}/\mf{m}^{n+1})(\gamma^2 \alpha \chi_p)) \to
H^1(G, (\ms{O}/\mf{m}^n)(\gamma^2 \alpha \chi_p)) \to
H^2(G, (\ms{O}/\mf{m})(\gamma^2 \alpha \chi_p)).
\end{displaymath}
Now, $(\ms{O}/\mf{m})(\gamma^2 \alpha \chi_p)$ is just $k(\ol{\alpha}
\cdot \ol{\chi}_p)$.
The group $H^2(G, k(\ol{\alpha} \cdot \ol{\chi}_p))$ is dual to $H^0(G,
k(\ol{\alpha}^{-1}))$ which vanishes since $\ol{\alpha} \ne 1$.  We thus see
that the map on $H^1$ groups above is surjective.  It follows that the map
\begin{displaymath}
H^1(G, \ms{O}(\gamma^2 \alpha \chi_p)) \to H^1(G, k(\ol{\alpha} \cdot
\ol{\chi}_p))
\end{displaymath}
is surjective.  Let $c$ be any class lifting $\ol{c}$ and let $\rho'$ be
the corresponding representation.   Then $\rho=\gamma^{-1} \otimes \rho'$
is a definite type $A$ lift of $\ol{\rho}$.

We now consider the case where $\ol{\rho}$ is irreducible.  Let $I^w$ be
the wild inertia subgroup of $G$.  As $I^w$ is pro-$p$ the space of
invariants $\ol{\rho}^{I^w}$ is non-zero.  Since $I^w$ is normal, this
space is stable under $G$.  It follows that it must be all of
$\ol{\rho}$.  We thus see that $\ol{\rho}(I^w)=1$, and so $\ol{\rho}$ may
be regarded as a representation of $G/I^w$.  As the tame inertia group $I^t$
is prime-to-$p$, we have $\ol{\rho} \vert_{I^t}=\ol{\alpha} \oplus \ol{\beta}$.
If $\ol{\alpha}=\ol{\beta}$ then $\ol{\alpha}$ extends to $G$ and after
twisting by $\ol{\alpha}^{-1}$ the
representation $\ol{\rho}$ would be unramified, and therefore not irreducible.
Thus $\ol{\alpha} \ne \ol{\beta}$.  Now, $G$ permutes the two
characters $\ol{\alpha}$ and $\ol{\beta}$.  If it permuted them trivially
then each would extend to all of $G$ and $\ol{\rho}$ would be reducible.
We thus see that $G$ permutes $\ol{\alpha}$ and $\ol{\beta}$ non-trivially.
Let $H$ be the subgroup of $G$ which fixes $\ol{\alpha}$.
The character $\ol{\alpha}$ extends to $H$ and $\ol{\rho}$ is the induction
of $\ol{\alpha}$ from $H$ to $G$.  Let $\psi$ be the character of $G_{F_v'}
\to \ol{\Q}_p^{\times}$ afforded by Lemma~\ref{lem:plift}.  Let $\gamma:
G_{F_v'} \to \ol{\Q}_p^{\times}$ be a finite order character lifting the
residual character $\ol{\alpha} \cdot \ol{\psi}^{-1}$.  Put $\rho=\Ind_H^G(
\gamma \psi)$.  It is clear that $\rho$ is a lift of $\ol{\rho}$.  Now,
$\rho \vert_H=(\gamma \psi) \oplus (\gamma' \psi')$ which shows that
$\rho$ is de Rham with non-positive Hodge-Tate weights.  As $\det{\rho}$
is a finite order character times the cyclotomic character, it follows
that $\rho$ is weight two.  Since $\ol{\rho}$ is irreducible, so is $\rho$,
which shows that $\rho$ is definite of type $B$.
\end{proof}

\begin{remark}
Proposition~\ref{prop:llift} and Proposition~\ref{prop:plift} show that any
residual representation is compatible with some type.  Note, however, that
it is not always possible to find a lift of a specified type.  For instance,
residual representations which are irreducible will not
admit lifts of type $C$.  When $v \nmid p$ it is not difficult to determine
exactly which types a given residual representation are compatible with.
When $v \mid p$ we have not worked this out completely.
\end{remark}

\subsection
We now prove Theorem~\ref{thm:lift2}.  By Proposition~\ref{prop:llift} and
Proposition~\ref{prop:plift} we can extend $t$ to a type function on all
of $\Sigma$ which is compatible with $\ol{\rho}$.  We may thus assume
$\Sigma=\Sigma'$.  For each $v \in \Sigma$ let $\rho_v'$ be a lift of
$\ol{\rho} \vert_{G_{F_v}}$ which is weight two for $v \mid p$ and has definite
type $t(v)$.  The existence of these lifts is assured by the compatibility of
$\ol{\rho}$ and $t$.  The character $\psi \chi_p \cdot (\det{\rho_v'})^{-1}$
is finite and pro-$p$.  Let
$\rho_v$ be the twist of $\rho_v'$ by the square root of this character.
Then $\rho_v$ is a lift of $\ol{\rho} \vert_{G_{F_v}}$, has weight two for
$v \mid p$, has definite type $t(v)$ and also has determinant $\psi \chi_p
\vert_{G_{F_v}}$.  Let $\tau_v$ be the
inertial type of $\rho_v$.  We have thus produced a locally solvable lifting
problem $\ms{P}=(\Sigma, \psi, t, \{\tau_v\})$.  Theorem~\ref{thm:lift} shows
that $\ms{P}$ has a solution, which gives the required lift.

\subsection
We close \S \ref{s:problems} by establishing the first statement of
Theorem~\ref{mainthm2}.  Precisely, we prove the following (see also
Remark~\ref{rem:noA1}):

\begin{proposition}
\label{prop:minlift}
Let $\ol{\rho}:G_F \to \GL_2(\ol{\F}_p)$ be an irreducible odd representation
satisfying (A1) and (A2).  Then there is a finitely ramified, weight two lift
$\rho:G_F \to \GL_2(\ol{\Q}_p)$ of $\ol{\rho}$ such that the prime-to-$p$
conductors of $\rho$ and $\ol{\rho}$ agree.
\end{proposition}

\begin{proof}
Write $\det{\ol{\rho}}=\ol{\psi} \cdot \ol{\chi}_p$ and let $\psi$ be the
Teichm\"uller lift of $\ol{\psi}$.  Let $\Sigma$ be the union of
the set of primes above $p$ and the set of primes at which $\ol{\rho}$
ramifies.  For each $v \nmid p$ in $\Sigma$ use Proposition~\ref{prop:llift}
to produce a definite type lift $\rho'_v:G_{F_v} \to \GL_2(\ol{\Q}_p)$ of
$\ol{\rho} \vert_{G_{F_v}}$ such that $\rho_v'$ and $\ol{\rho}_v$ have the
same conductor.  Let $\rho_v$ be the twist of $\rho'_v$ by the square root of
$\psi \chi_p \cdot (\det{\rho'_v})^{-1}$, so that $\rho_v$ has determinant
$\psi \chi_p$.  One easily sees that $\rho_v$ has the same conductor as
$\rho'_v$.  For $v \mid p$ use Proposition~\ref{prop:plift} to produce a
definite type weight two lift $\rho_v'$ of $\ol{\rho} \vert_{G_{F_v}}$ and let
$\rho_v$ be an appropriate twist with determinant $\psi \chi_p$.  For $v \in
\Sigma$ let $t(v)$ be the type of $\rho_v$ and let $\tau_v$ be the inertial
type of $\rho_v$.  We have thus defined a lifting problem $\ms{P}=(\Sigma,
\psi, t, \{\tau_v\})$ which is locally solvable.  Let $\rho$ be a solution to
$\ol{\rho}$.  For $v \nmid p$ in $\Sigma$ the representations $\rho
\vert_{G_{F_v}}$ and $\rho_v$ have the same type and inertial type and so
$\rho \vert_{I_{F_v}} \cong \rho_v \vert_{I_{F_v}}$.  This shows that $\rho$
has the same prime-to-$p$ conductor as $\ol{\rho}$.
\end{proof}

\section{Two additional results}
\label{s:additional}

\subsection
We now give two additional results, building on the main theorems we
have proved.  The first could be
called ``solvable descent for mod $p$ Hilbert eigenforms.''  Recall
that Langlands' base change implies that if $\rho$ is a $p$-adic representation
of $G_F$ and $F'/F$ is a solvable extension then $\rho$ is modular if and
only if $\rho \vert_{G_{F'}}$ is.  If $\ol{\rho}$ is a residual representation
of $G_F$ then the modularity of $\ol{\rho}$ clearly implies that of
$\ol{\rho} \vert_{G_{F'}}$ by the previous sentence.  However, the other
direction --- descent --- is not so clear.  Khare \cite{Khare} has established
some results in the case $F=\Q$.  We prove the following:

\begin{proposition}
\label{prop:descent}
Let $\ol{\rho}:G_F \to \GL_2(\ol{\F}_p)$ be a representation satisfying
(A1) and (A2) and let $t:\Sigma_p \to \{A,B,C\}$ be a type function compatible
with $\ol{\rho}$.  Assume there exists
a finite, solvable, totally real extension $F'/F$ and a parallel weight two
Hilbert eigenform $f$ over $F'$ with coefficients in $\ol{\Q}_p$ such that
$\ol{\rho} \vert_{G_{F'}}$ still satisfies (A1) and (A2) and we have
$\ol{\rho} \vert_{G_{F'}} \cong \ol{\rho}_f$ and $t_f \vert_{\Sigma_p}=
t \vert_{F'}$.
Then there exists a parallel weight two Hilbert eigenform $g$ over $F$ with
coefficients in $\ol{\Q}_p$ such that $\ol{\rho} \cong \ol{\rho}_g$
and $t_g=t$.
\end{proposition}

\begin{proof}
By Theorem~\ref{thm:lift2} we can find a lift $\rho$ of $\ol{\rho}$
such that $t_{\rho} \vert_{\Sigma_p}=t$.  Then $\rho \vert_{G_{F'}}$ and
$\rho_f$ have the same type above $p$ and isomorphic reduction and so
Theorem~\ref{thm:mlt} implies the modularity
of $\rho \vert_{G_{F'}}$.  Langlands' base change now implies the modularity
of $\rho$ and thus of $\ol{\rho}$.
\end{proof}

\subsection
Our second result is a strengthening of our potential modularity theorem.

\begin{proposition}
\label{prop:pmod3}
Let $S$ be a finite set of places of $F$.  Then in 
Theorem~\ref{thm:pmod2} one can assume that the extension $F'/F$ splits at
all places in $S$.  The same conclusion holds in Theorem~\ref{thm:pmod1}
under the additional hypothesis that for $v \in S \cap \Sigma_p$ the
type $t(v)$ is compatible with $\ol{\rho} \vert_{G_{F_v}}$.
\end{proposition}

To prove this we use a combination of solvable descent and the following
modification of the theorem of Moret-Bailly.

\begin{proposition}
Let $(X, \Sigma, \{L_v\}, \{\Omega_v\})$ be a Skolem datum (see \S \ref{ss:mb})
and let $S$ be a finite set of places of $F$.  One can then find a finite
solvable extension $F_1/F$ which splits over each $L_v$, a finite extension
$F_2/F$ which splits at all places in $S$ and over each $L_v$ and is linearly
disjoint from $F_1$ and a point $x \in X(F_1 F_2)$ such that for each embedding
$F_1 F_2 \to L_v$ the image of $x$ in $X(L_v)$ belongs to $\Omega_v$.
\end{proposition}

\begin{proof}
We are free, of course, to enlarge $S$.  We thus assume that $S$ contains
$\Sigma$ and write $S=\Sigma \cup \Sigma'$.  Let $F_1/F$ be any
finite, solvable extension such that $F_{1, w}=L_v$ for $w$ lying over
$v \in \Sigma$ and $X(F_{1, w})$ is non-empty for $w$ lying over $v \in
\Sigma'$.  Let $X'$ be the restriction of scalars from $F_1$ to $F$ of
$X_{F_1}$.  This is smooth and geometrically connected since its base change
to $\ol{F}$ is isomorphic to a product of copies of $X_{\ol{F}}$.
We have $X'(F_v)=\prod_{w \mid v} X(F_{1, w})$ for any place $v$ of
$F$.  We define a Skolem datum $(X', S, \{L_v'\}, \{\Omega_v'\})$ by taking
$L_v'=F_v$ for $v \in S$ and taking $\Omega_v'$ be the product of the
$\Omega_v$'s for $v \in \Sigma$ and all of $X'(F_v)$ for $v \in \Sigma'$.  By
Proposition~\ref{prop:mb2} we can find an extension $F_2/F$ which is
linearly disjoint from $F_1$ and which splits at $S$ and a point $x \in
X'(F_2)$ such that for each $v \in S$ and each embedding $F_2 \to F_v$ the
image of $x$ in $X'(F_v)$ lands in $\Omega'_v$.  Since $F_2$ is linearly
disjoint from $F_1$ we have $X'(F_2)=X(F_1 F_2)$ and so $x$ can be regarded as
a point in the latter set.  The proposition easily follows.
\end{proof}

We now go back to the proof of Proposition~\ref{prop:pmod3}.

\begin{proof}[Proof of Proposition~\ref{prop:pmod3}]
We only sketch a proof.  The idea is to use the above proposition in
place of the original theorem of Moret-Bailly to get an abelian variety
$A/F_1 F_2$ with $A[\mf{p}]=\ol{\rho} \vert_{G_{F_1 F_2}}$ and
$A[\mf{l}]=\ol{\rho}' \vert_{G_{F_1 F_2}}$.  This gives the modularity of
$\ol{\rho}$ over $F_1 F_2$ as in the proof of Proposition~\ref{thm:pmod1}.
One then uses solvable descent with respect to the extension $F_1 F_2/F_2$
to conclude that $\ol{\rho}$ is modular over $F_2$.  (The field $F'$ is
then just $F_2$.)
\end{proof}

\section{Consequences of potential modularity}
\label{s:conseq}

\subsection
We now establish the four statements in Corollary~\ref{maincor}.  As remarked
in the introduction, many of these proofs are well-known or exist already in
the literature; we feel, though, that it would be useful to have them in one
place.  We begin with the first statement.

\begin{proposition}
\label{prop:conseq1}
Let $\rho:G_F \to \GL_2(\ol{\Q}_p)$ be a finitely ramified, odd, weight two
representation such that $\ol{\rho}$ satisfies (A1) and (A2).  If
$v \nmid p$ is a place of $F$ at which $\rho$ is unramified then the
eigenvalues of $\rho(\Frob_v)$ belong to $\ol{\Q} \subset \ol{\Q}_p$ and under
any embedding into $\C$ have modulus $(\bN{v})^{1/2}$.
\end{proposition}

\begin{proof}
It is easy to see that if there exists a finite extension $F'/F$ such that
$\rho \vert_{G_{F'}}$ satisfies the conclusion of the proposition then so
does $\rho$.  By Theorem~\ref{thm:pmod2} it therefore suffices to prove the
theorem when $\rho$ is associated to a Hilbert eigenform.  This has been
established by Blasius \cite{Blasius}, extending the earlier work of
Brylinski and Labesse \cite{BrylinskiLabesse}.
\end{proof}

\subsection
We now prove statement (2) of Corollary~\ref{maincor}.  We learned this
proof from a lecture given by Taylor at the Summer School on Serre's Conjecture
held at Luminy in 2007.  Taylor attributed the proof to Dieulefait; a sketch
of the argument can be found in \cite[\S 3.2]{Dieulefait}.

\begin{proposition}
Let $\rho:G_F \to \GL_2(\ol{\Q}_p)$ be a finitely ramified, odd, weight two
representation such that $\ol{\rho}$ satisfies (A1) and (A2).  Then there
exists a number field $K$, a place $v_0 \mid p$ of $K$, an embedding $K_{v_0}
\to \ol{\Q}_p$ and a compatible system $\{\rho_v:G_F \to \GL_2(K_v)\}$ indexed
by the places of $K$ such that $\rho$ is equivalent to $\rho_{v_0}
\otimes_{K_{v_0}} \ol{\Q}_p$.  Each representation $\rho_v$ is finitely
ramified, odd, weight two and absolutely irreducible.
\end{proposition}

\begin{proof}
Apply Theorem~\ref{thm:pmod1} to produce a finite Galois totally real
extension $F'/F$ linearly disjoint from $\ker{\ol{\rho}}$ and a parallel
weight two Hilbert eigenform $f$ over $F'$ with coefficients in $\ol{\Q}_p$
such that $\rho \vert_{G_{F'}}=\rho_f$.  Let $I$ be the set of fields $F''$
which are intermediate to $F'$ and $F$ and for which $\Gal(F'/F'')$ is
solvable.  For $i \in I$ we write $F_i$ for the corresponding field.  For
each $i$ we can use solvable descent to find a parallel weight two cuspidal
Hilbert eigenform $f_i$ with coefficients in $\ol{\Q}_p$ such that $\rho
\vert_{G_{F_i}}=\rho_{f_i}$.  Let $K_i$ denote the field of coefficients of
$f_i$; note that this comes with a given embedding $K_i \to \ol{\Q}_p$.
Let $K$ be a number field which is Galois over $\Q$, into which each $K_i$
embeds and which contains
all roots of unity of order $[F':F]$.  Fix an embedding $K \to \ol{\Q}_p$ and
embeddings $K_i \to K$ such that the composite $K_i \to K \to \ol{\Q}_p$
is the given embedding.  Let $v_0$ be the place of $K$ determined by the
embedding $K \to \ol{\Q}_p$.
For each place $v$ of $K$ and each $i \in I$ we have a representation
$r_{i, v}:G_{F_i} \to \GL_2(K_v)$ associated to the Hilbert form $f_i$.
It is absolutely irreducible.
Note that after composing $r_{i, v_0}$ with the embedding $\GL_2(K_{v_0})
\to \GL_2(\ol{\Q}_p)$ we obtain $\rho \vert_{G_{F_i}}$.

By Brauer's theorem, we can write
\begin{displaymath}
1=\sum_{i \in I} n_i \Ind_{\Gal(F'/F_i)}^{\Gal(F'/F)}(\chi_i)
\end{displaymath}
where the $n_i$ are integers (possibly negative) and the $\chi_i$ are
characters of $\Gal(F'/F_i)$ valued in $K^{\times}$.  (Here we use the fact
that $K$ contains all roots of unity of order $[F':F]$.)  This equality is
taken in the Grothendieck group of representations of $\Gal(F'/F)$ over $K$.
Note that by taking the dimension of each side we find $\sum n_i
[F_i:F]=1$.

Let $v$ be a place of $K$.  For a number field $M$ write $\mc{C}_{M, v}$ for
the category of semi-simple continuous representations of $G_M$ on finite
dimensional $K_v$-vector spaces.  The category $\mc{C}_{M, v}$ is a
semi-simple abelian category.  We let $K(\mc{C}_{M, v})$ be its Grothendieck
group.  It is the free abelian category on the set of irreducible continuous
representations of $G_M$ on $K_v$-vector spaces.  We let $(,)$ be the
integer valued pairing on $K(\mc{C}_{M, v})$ given by $(A, B)=\dim_{K_v}
\Hom(A, B)$.  This is well-defined because $\mc{C}_{M, v}$ is semi-simple.
It is symmetric.  If $M'/M$ is a finite extension then we have
adjoint functors $\Ind_{M'}^M:\mc{C}_{M', v} \to \mc{C}_{M, v}$ and
$\Res^M_{M'}: \mc{C}_{M, v} \to \mc{C}_{M', v}$.  (One must check, of course,
that induction and restriction preserve semi-simplicity --- we leave this to
the reader.)  These functors induce
maps on the $K$-groups which are adjoint with respect to $(,)$.  If $M_1$
and $M_2$ are two extensions of $M$ and $r_1$ belongs to $\mc{C}_{M_1, v}$
and $r_2$ belongs to $\mc{C}_{M_2, v}$ then we have the formula
\begin{equation}
\label{eq4}
(\Ind_{M_1}^M(r_1), \Ind_{M_2}^M(r_2))=
\sum_{g \in S} (\Res^{M_1^g}_{M_1^g M_2} (r_1^g), \Res^{M_2}_{M_1^g M_2} (r_2))
\end{equation}
where $S$ is a set of representatives for $G_{M_1} \bs G_M / G_{M_2}$,
$M_1^g$ is the field determined by $g G_{M_1} g^{-1}$ and
$r_1^g$ is the representation of $g G_{M_1} g^{-1}$ given by $x \mapsto
r_1(g^{-1} x g)$.  This formula is gotten by using Frobenius reciprocity
and Mackey's formula.

Define
\begin{displaymath}
\rho_v=\sum_{i \in I} n_i \Ind_{F_i}^F(r_{i, v} \otimes \chi_i),
\end{displaymath}
which is regarded as an element of $K(\mc{C}_{F, v})$.  We now show that
each $\rho_v$ is (the class of) an absolutely irreducible two dimensional
representation.  To begin with, we have
\begin{displaymath}
\begin{split}
\rho_{v_0} \otimes_{K_{v_0}} \ol{\Q}_p
=& \sum_{i \in I} n_i \Ind_{F_i}^F((r_{i, v_0} \otimes_{K_{v_0}} \ol{\Q}_p)
\otimes_K \chi_i) \\
=& \sum_{i \in I} n_i \Ind_{F_i}^F((\rho \vert_{F_i}) \otimes_K \chi_i) \\
=& \bigg[ \sum_{i \in I} n_i \Ind_{F_i}^F(\chi_i) \bigg] \otimes_K \rho \\
=& \rho
\end{split}
\end{displaymath}
This shows that $\rho_{v_0}$ is (the class of) an absolutely irreducible
representation.

Now let $v$ be an arbitrary finite place of $K$.  We have
\begin{displaymath}
\begin{split}
(\rho_v, \rho_v)
=& \sum_{i, j \in I} n_i n_j (\Ind_{F_i}^F(r_{i, v} \otimes \chi_i),
\Ind_{F_j}^F(r_{j, v} \otimes \chi_j)) \\
=& \sum_{i, j \in I} \sum_{g \in S_{ij}} n_i n_j (\Res^{F_i^g}_{F_i^g F_j}(
(r_{i, v} \otimes \chi_i)^g), \Res^{F_j}_{F_i^g F_j}(r_{j, v} \otimes \chi_j))
\end{split}
\end{displaymath}
where we have used \eqref{eq4}.  Here $S_{ij}$ is a set of representatives
for $G_{F_1} \bs G_F / G_{F_2}$.  The representation $r_{i, v} \vert_{F'}$
is the representation coming from the cusp form $f'$ and so is absolutely
irreducible.  It follows that the restriction of $r_{i, v}$ to any subfield
of $F'$ is absolutely irreducible.  Thus the representations occurring in
the pairing in the second line above are irreducible.  It follows that the
pairing is then either 1 or 0 if the representations are isomorphic or not.  
Therefore, if let $\delta_{v, i, j, g}$ be 1 or 0 according to whether
$\Res^{F_i^g}_{F_i^g F_2}(r_{i, v} \otimes \chi_i)^g$ is isomorphic to
$\Res^{F_j}_{F_i^g F_2}(r_{j, v} \otimes \chi_j)$ then we find
\begin{displaymath}
(\rho_v, \rho_v)=\sum_{i, j \in I} \sum_{g \in S_{ij}} n_i n_j \delta_{v, i,j
,g}.
\end{displaymath}
Now, the $\{r_{i, v}\}_v$ and the $\{r_{j, v}\}_v$ form a compatible system.
It follows that $\delta_{v, i, j, g}$ is independent of $v$.  The above
formula thus gives
\begin{displaymath}
(\rho_v, \rho_v)=(\rho_{v'}, \rho_{v'})
\end{displaymath}
if $v'$ is another place of $K$.  Taking $v'=v_0$ and using that $\rho_{v_0}$
is an absolutely irreducible representation gives $(\rho_v, \rho_v)=1$.
Now, if we write $\rho_v=\sum m_i \pi_i$ where $m_i \in \Z$ and the $\pi_i$
are mutually
non-isomorphic irreducible representations then we have $(\rho_v, \rho_v)=
\sum m_i^2 (\pi_i, \pi_i)$.  Since the terms are all non-negative integers
and the sum is 1, we find $\rho_v=\pm \pi$ with $(\pi, \pi)=1$.  Thus
$\pi$ is an absolutely irreducible representation.  Now,
$\dim{\rho_v}=2$ since each $r_{i, v}$ is two dimensional and
$\sum n_i [F_i:F]=1$.  Since $\dim{\pi}$ is non-negative, we must have
$\rho_v=\pi$.  This proves that $\rho_v$ is the class of an absolutely
irreducible representation.

We must now show that $\rho_v$ is finitely ramified, odd, weight two and
compatible with $\rho_{v_0}$.  That $\rho_v$ is finitely ramified follows
immediately from the definition.  It also follows immediately from the
definition that $\rho_v \vert_{F'}$ is equivalent to $\rho_{f', v}$.  This
shows that $\rho_v$ is odd and weight two.  To show that $\rho_v$ is
compatible with $\rho_{v_0}$ look at the traces of Frobenii in $\rho_v$;
we leave the details to the reader.
\end{proof}

\begin{remark}
The compatible system constructed above is in fact \emph{strongly
compatible}.  For a discussion of this, see \cite[Theorem~6.6]{Taylor}.
\end{remark}

\subsection
We now discuss statement (3) of the corollary.  We prefer not to give a formal
proof so as to avoid giving formal definitions for $L$-functions and their
functional equations.  For details, see \cite[Corollary~2.2]{Taylor3} or
\cite[Theorem~6.6]{Taylor}:  the following sketch is simply a paraphrasing
of the arguments there.  Let $\rho:G_F \to \GL_2(\ol{\Q}_p)$ be given.
Pick an extension $F'/F$ such that $\rho \vert_{G_{F'}}$ is modular.  Using
notation as in the previous section, write
\begin{displaymath}
1=\sum_{i \in I} n_i \Ind_{F_i}^F(\chi_i)
\end{displaymath}
so that
\begin{displaymath}
\rho=\sum_{i \in I} n_i \Ind_{F_i}^F(r_i \otimes \chi_i).
\end{displaymath}
Taking $L$-functions gives
\begin{displaymath}
L(s, \rho)=\prod_{i \in I} L(r_i \otimes \chi_i, s)^{n_i}.
\end{displaymath}
Since the representations $r_i$ are modular, the $L$-functions on the right
side have meromorphic continuation and functional equations.  It follows that
the same holds for the $L$-function on the left side.  We also claimed in (3)
that $L(s, \rho)$ converges for $\re{s} > \sfrac{3}{2}$ (by which we mean,
of course, the Euler product defining $L(s, \rho)$ converges in this
region); this follows immediately from the bounds given in (1).

\subsection
We now turn to claim (4) of the corollary.  We follow
\cite[Corollary~2.4]{Taylor3} in our approach.  The precise statements we
prove are the following:

\begin{proposition}
\label{prop:astate}
Let $\rho:G_F \to \GL_2(\ol{\Q}_p)$ be a finitely ramified, odd, weight two
representation such that $\ol{\rho}$ satisfies (A1) and (A2).  Assume
furthermore that $F$ has odd degree or that there exists a place $v \nmid p$
for which $\rho \vert_{G_{F_v}}$ is indecomposable.  Then there exists a number
field $K$, an embedding $K \to \ol{\Q}_p$ and a $\GL_2(K)$-type abelian
variety $A/F$ such that $\rho$ is equivalent to $T_p A \otimes_{\ms{O}_K}
\ol{\Q}_p$.
\end{proposition}

\begin{proposition}
\label{prop:astors}
Let $\ol{\rho}:G_F \to \GL_2(\ol{\F}_p)$ be an irreducible odd representation.
If $p=5$ and the projective image of $\ol{\rho}$ is $\PGL_2(\F_5)$ then
assume $[F(\zeta_p):F]=4$.  Then there exists a number
field $K$, a prime $\mf{p}$ of $K$ lying above $p$ and a $\GL_2(K)$-type
abelian variety $A/F$ such that $\ol{\rho}$ is equivalent to $A[\mf{p}]$.
\end{proposition}

We prove Proposition~\ref{prop:astate} by combining the following statements:
(1) $\rho$ is potentially modular; (2) the conclusion of the proposition is
valid for modular representations (see Lemma~\ref{lem:av2} below); and (3)
the conclusion of the proposition can be checked potentially (see
Lemma~\ref{lem:av} below).  We prove
Proposition~\ref{prop:astors} by using our previous results to produce an
appropriate lift of the residual representation and then applying
Proposition~\ref{prop:astate} (although a separate argument is used when
(A1) does not hold).   Note that Proposition~\ref{prop:astors}
gives statement (2) of Theorem~\ref{mainthm2}.

In what follows, we say that a representation $\rho:G_F \to \GL_2(\ol{\Q}_p)$
(resp.\ a representation $\ol{\rho}:G_F \to \GL_2(\ol{\F}_p)$)
\emph{comes from a $\GL_2$-type abelian variety} if there is a number field
$K$, an embedding $K \to \ol{\Q}_p$ and a $\GL_2(K)$-type abelian variety
$A/F$ such that $\rho$ is equivalent to $T_pA \otimes_{\ms{O}_K} \ol{\Q}_p$
(resp.\ such that $\ol{\rho}$ is equivalent to $A[\mf{p}] \otimes_{k_{\mf{p}}}
\ol{\F}_p$, where $\mf{p}$ is the prime determined by $K \to \ol{\Q}_p$,
$k_{\mf{p}}$ is its residue field and $k_{\mf{p}} \to \ol{\F}_p$ is the
embedding determined by $K \to \ol{\Q}_p$).
The following lemma says that modular representations often satisfy this
condition:

\begin{lemma}
\label{lem:av2}
Let $F$ be a totally real field and let $f$ be a parallel weight two
cuspidal Hilbert eigenform over $F$ with coefficients in $\ol{\Q}_p$.  Assume
that $F$ has odd degree or that at some finite place $v \nmid p$
the form $f$ (or more accurately, the associated automorphic representation)
is square-integrable.  Then $\rho_f$ comes from a $\GL_2$-type abelian.
\end{lemma}

\begin{proof}
Due to the hypotheses, one can use the Jacquet-Langlands correspondence to
transfer $f$ to a quaternion algebra $B$ over $F$ which splits at precisely
one of the infinite places.  The result then follows from
\cite[Theorem~4.12]{Hida}, where $\rho_f$ is found in the Jacobian of the
Shimura curve associated to $B$.
\end{proof}

We now show that one can check if a representation comes from a $\GL_2$-type
abelian variety by passing to a finite extension.  The proof
of the following lemma comes directly from \cite[Corollary~2.4]{Taylor3}.

\begin{lemma}
\label{lem:av}
Let $F$ be a number field and let $\rho:G_F \to \GL_2(\ol{\Q}_p)$ be an
irreducible representation.  Assume there exists a finite extension $F'/F$
such that $\rho \vert_{G_{F'}}$ is irreducible and comes from a $\GL_2$-type
abelian variety.  Then $\rho$ comes from a $\GL_2$-type abelian variety.
\end{lemma}

\begin{proof}
Pick a number field $K'$, an embedding $K' \to \ol{\Q}_p$ and a
$\GL_2(K')$-type abelian variety $A'/F'$ such that $\rho \vert_{G_{F'}}$ is
equivalent to $T_pA' \otimes_{K'} \ol{\Q}_p$.  We have
\begin{displaymath}
\End_{F, K'}(\Res^{F'}_F{A'})=\bigoplus_{i \in I} K_i'.
\end{displaymath}
Here $\End_{F, K'}$ denotes endomorphisms in the isogeny category of
$\GL_2(K')$-type abelian varieties over $F$, $\Res^{F'}_F{A'}$ denotes the
restriction of scalars of $A'$ from $F'$ to $F$, $I$ is some finite index
set and each $K_i'$ is a simple $K'$-algebra corresponding to some simple
direct factor $A_i'$ of $\Res^{F'}_F{A'}$.  Tensoring the above with
$\ol{\Q}_p$, we find
\begin{displaymath}
\End_{F, K'}(\Res^{F'}_F{A'}) \otimes_{K'} \ol{\Q}_p
=\bigoplus_{i \in I} (K_i' \otimes_{K'} \ol{\Q}_p)
\end{displaymath}
On the other hand, we have
\begin{displaymath}
\begin{split}
\End_{F, K'}(\Res^{F'}_F{A'}) \otimes_{K'} \ol{\Q}_p
=& \End_{\ol{\Q}_p[G_F]}(T_p(\Res^{F'}_F{A'}) \otimes_{\ms{O}_{K'}}
\ol{\Q}_p) \\
=& \End_{\ol{\Q}_p[G_F]}(\Ind_{F'}^F(T_pA' \otimes_{\ms{O}_{K'}} \ol{\Q}_p)) \\
=& \End_{\ol{\Q}_p[G_F]}(\Ind_{F'}^F(\rho \vert_{G_{F'}})) \\
\end{split}
\end{displaymath}
The first step above is Faltings' theorem; the other steps are straightforward.
Since $\rho \vert_{G_{F'}}$ is irreducible the induction $\Ind_{F'}^F(
\rho \vert_{G_{F'}})$ is semi-simple.  Furthermore, we have
\begin{displaymath}
\Hom_{\ol{\Q}_p[G_F]}(\Ind_{F'}^F(\rho \vert_{G_{F'}}), \rho)
=\Hom_{\ol{\Q}_p[G_{F'}]}(\rho \vert_{G_{F'}}, \rho \vert_{G_{F'}})
=\ol{\Q}_p.
\end{displaymath}
It follows that $\rho$ has multiplicity one in $\Ind_{F'}^F(\rho
\vert_{G_{F'}})$.  Thus $\End_{\ol{\Q}_p[G_F]}(\Ind_{F'}^F(\rho
\vert_{G_{F'}}))$ has a canonical $\ol{\Q}_p$ occurring inside of it as a
direct factor, namely $\rho$-isotypic projector.
We thus find a canonical
$\ol{\Q}_p$ occurring as a summand in $\End_{F, K'}(\Res^{F'}_F{A'})
\otimes_{K'} \ol{\Q}_p$.  This factor must appear as a factor of $K_i'
\otimes_{K'} \ol{\Q}_p$, for a unique $i \in I$.  The simple $K'$-algebra
$K_i'$ must therefore be a field.  We take $K=K_i'$, $A=A_i'$ and use the
canonical $\ol{\Q}_p$ factor of $K \otimes_{K'} \ol{\Q}_p$ to define our
embedding $K \to \ol{\Q}_p$.  It
follows directly that $\rho$ is equivalent to $T_pA \otimes_{\ms{O}_K}
\ol{\Q}_p$ and thus comes from an abelian variety.
\end{proof}

We now prove Proposition~\ref{prop:astate}.

\begin{proof}[Proof of Proposition~\ref{prop:astate}]
Let $\rho$ be given.  First consider the case where $F$ has odd degree.
Use Theorem~\ref{thm:pmod2} to produce a finite, totally real, Galois
extension $F''/F$ which is linearly disjoint from $\ker{\ol{\rho}}$ and a
parallel weight two Hilbert eigenform $f''$ with coefficients in $\ol{\Q}_p$
such that $\rho \vert_{G_{F''}} \cong \rho_{f''}$.  Let $G$ be the Galois
group of $F''/F$, let $H$ be its 2-Sylow subgroup and let $F'$ be the
fixed field of $H$.  The extension $F''/F'$ is Galois with group $H$ and the
group $H$ is solvable since it is a 2-group.  Thus by solvable descent,
we can find a parallel weight two Hilbert modular form $f'$ over $F'$
with coefficients in $\ol{\Q}_p$ such that $\rho \vert_{G_{F'}} \cong
\rho_{f'}$.  The field $F'$ has odd degree over $F$ and thus odd degree
over $\Q$.  Lemma~\ref{lem:av2} now shows that $\rho \vert_{G_{F'}}$ comes
from an abelian variety.  Lemma~\ref{lem:av} now gives that $\rho$ comes
from an abelian variety.

Now consider the case where $F$ has even degree and $\rho$ is indecomposable
at some finite place $v \nmid p$.  Apply Proposition~\ref{prop:pmod3} to
produce an extension $F'/F$ which is linearly disjoint from $\ker{\ol{\rho}}$
and in which $v$ splits completely and a parallel weight two Hilbert eigenform
$f'$ with coefficients in $\ol{\Q}_p$ such that $\rho \vert_{G_{F'}} \cong
\rho_{f'}$.  The form $f'$ is square integrable at $v$.  Lemma~\ref{lem:av2}
thus shows that $\rho \vert_{G_{F'}}$ comes from an abelian variety.
Lemma~\ref{lem:av} now gives that $\rho$ comes from an abelian variety.
\end{proof}

We now prove Proposition~\ref{prop:astors}.

\begin{proof}[Proof of Proposition~\ref{prop:astors}]
Let $\ol{\rho}$ be given.  First assume that $\ol{\rho}$ satisfies (A1).  Pick
a place $v \nmid p$ of $F$ such that $\ol{\chi}_p \vert_{G_{F_v}}$ and
$\ol{\rho} \vert_{G_{F_v}}$ are trivial.  One easily sees that there is a lift
of $\ol{\rho} \vert_{G_{F_v}}$ which is non-crystalline and of the form
\begin{displaymath}
\mat{\chi_p}{\ast}{}{1}
\end{displaymath}
Thus $\ol{\rho}$ is compatible with type $C$ at $v$.  Therefore, using
Theorem~\ref{thm:lift2} we can find a finitely ramified, odd, weight two
lift $\rho$ of $\ol{\rho}$ which is type $C$ (and thus indecomposable) at $v$.
Proposition~\ref{prop:astate} shows that $\rho$ comes from an abelian
variety, and so $\ol{\rho}$ does as well.

Now assume that $\ol{\rho}$ does not satisfy (A1).  One then finds $\ol{\rho}=
\Ind_{F'}^F(\ol{\alpha})$ where $F'$ is the quadratic extension of $F$
contained in $F(\zeta_p)$ and $\ol{\alpha}:G_{F'} \to \ol{\F}_p^{\times}$ is
some character.  Let $\alpha$ be the Teichm\"uller lift of $\ol{\alpha}$ and
put $\rho_0=\Ind_{F'}^F(\alpha)$.  Then $\rho_0$ is a lift of $\ol{\rho}$.
Furthermore, $\rho_0$ is modular (since it is an induction) and has Hodge-Tate
weights zero (since it has finite image).  Thus $\rho_0$
comes from a parallel weight one Hilbert eigenform $f_0$ over $F$.  Multiply
$f_0$ by a high weight modular form $g$ congruent to 1 modulo $p$ (of the sort
provided by \cite[Lemma~1.4.2]{Wiles2}) and then find a congruence between the
resulting form and a
parallel weight two eigenform $f$.  If $F$ has odd degree then $\rho_f$ comes
from an abelian variety by Lemma~\ref{lem:av2}, and so $\ol{\rho}$ does as
well.  If $F$ has even degree, transfer $f$ to the
quaternion algebra over $F$ ramified at the infinite places and use
Lemma~\ref{lem:raise} to find a congruence between $f$ and a form $f'$ which
is special at some finite place prime to $p$.  By
Lemma~\ref{lem:av2}, we find that $\rho_{f'}$ comes from an abelian variety,
and so $\ol{\rho}$ does as well.  This completes the proof.
\end{proof}

\begin{remark}
\label{rem:noA1}
Let $\ol{\rho}:G_F \to \GL_2(\ol{\F}_p)$ be an odd irreducible representation
not satisfying (A1).  The representation
$\rho_f$ constructed in the above proof gives a finitely ramified weight two
lift of $\ol{\rho}$.  In many cases, one can take the form $g$ to have level
divisible only by $p$; one can then take $f$ so that $\rho_f$ has the same
prime-to-$p$ conductor as $\ol{\rho}$.  Thus, in such cases, we can remove
the additional hypothesis from the first statement of Theorem~\ref{mainthm2}.
\end{remark}

\end{document}